\documentclass[11pt]{article}

\usepackage{amsmath, amssymb, amsthm}
\usepackage{enumerate}
\usepackage{fullpage}

\newtheorem{theorem}{Theorem}[section]
\newtheorem{lemma}[theorem]{Lemma}
\newtheorem{claim}[theorem]{Claim}
\newtheorem{problem}[theorem]{Problem}
\newtheorem{definition}[theorem]{Definition}
\newtheorem{corollary}[theorem]{Corollary}

\newcommand{\pr}[1]{\mathbb{P}\left[ #1 \right]}
\newcommand{\e}{\epsilon}
\newcommand{\E}[1]{\mathbb{E}\left[ #1 \right]}
\renewcommand{\SS}{\mathbb{S}}
\newcommand{\RR}{\mathbb{R}}
\newcommand{\bin}[1]{\text{\rm Bin}\left( #1 \right)}

\DeclareMathOperator{\poly}{poly}

\begin{document}

\title{The critical window for the classical Ramsey-Tur\'an problem}

\author{Jacob Fox\thanks{Department of Mathematics, MIT, Cambridge, MA
02139-4307. Email: {\tt fox@math.mit.edu}. Research supported by a Simons
Fellowship, NSF grant DMS-1069197, and the MIT NEC Research Corporation
Fund.}
\and
Po-Shen Loh\thanks{ Department of Mathematical Sciences,
Carnegie Mellon University, Pittsburgh, PA 15213.  E-mail: {\tt
ploh@cmu.edu}. Research supported by NSF grant DMS-1201380, an NSA
Young Investigators Grant and a USA-Israel BSF Grant.}
\and
Yufei Zhao \thanks{Department of
Mathematics, MIT, Cambridge, MA 02139-4307. Email: {\tt
yufeiz@math.mit.edu}. Research supported by an Akamai Presidential
Fellowship.}}

\date{}

\maketitle

\begin{abstract}

The first application of  Szemer\'edi's powerful regularity method was the
  following celebrated Ramsey-Tur\'an result proved by Szemer\'edi in 1972:
  any $K_4$-free graph on $n$ vertices with independence number $o(n)$ has
  at most $(\frac18+o(1))n^2$ edges. Four years later, Bollob\'as and Erd\H{o}s gave a surprising geometric construction, utilizing the isoperimetric inequality for the high dimensional sphere, of a $K_4$-free
  graph on $n$ vertices with independence number $o(n)$ and
  $(\frac18-o(1))n^2$ edges.  Starting with Bollob\'as and Erd\H{o}s in 1976, several problems have been asked on estimating the minimum possible independence number in the critical window, when the number of edges is about $n^2/8$.  These problems have received considerable attention and remained one of the main open problems in this area.  In this paper, we give nearly best-possible bounds, solving the various open problems concerning this critical window.

\end{abstract}

\section{Introduction}

Szemer\'edi's regularity lemma \cite{Sz78} is one of the most powerful tools in extremal combinatorics. Roughly speaking, it says that every graph can be partitioned into a small number of parts such that the bipartite subgraph between almost every pair of parts is random-like. The small number of parts is at most   an integer $M(\epsilon)$ which depends only on an approximation parameter $\epsilon$. The exact statement of the regularity lemma is given in the beginning of Section \ref{sec:originalproof}. For more
background on the regularity lemma, the interested reader may consult the
well-written surveys by Koml\'os and Simonovits \cite{KS} and R\"odl and
Schacht \cite{RoSc}.

In the regularity lemma, $M(\epsilon)$ can be taken to be a tower of twos
of height $\epsilon^{-O(1)}$, and probabilistic constructions of Gowers
\cite{Go97} and Conlon and Fox \cite{CoFo} show that this is best possible.
Unfortunately, this implies that the bounds obtained by applications of the
regularity lemma are usually quite poor. It remains an important problem to
determine if new proofs giving better quantitative estimates for certain
applications of the regularity lemma exist (see, e.g., \cite{Go00}).  Some
progress has been made, including the celebrated proof of Gowers
\cite{Go01} of Szemer\'edi's theorem using Fourier analysis, the new proofs
\cite{Co,CFS,FS09,GRR} that bounded degree graphs have linear Ramsey
numbers, the new proof \cite{Fo} of the graph removal lemma, and the new
proofs \cite{CBK,LSS} of P\'osa's conjecture for graphs of large
order.

The earliest application of the regularity method\footnote{We remark that
Szemer\'edi \cite{Sz-orig} first developed regularity lemmas which were
weaker than what is now commonly known as Szemer\'edi's regularity lemma as
stated above. Original proofs of several influential results, including
Theorem \ref{thm:full-reg}, Szemer\'edi's theorem \cite{Sz75} on long
arithmetic progressions in dense subsets of the integers, and the
Ruzsa-Szemer\'edi theorem \cite{RuSz} on the $(6,3)$-problem, used
iterative applications of these weak regularity lemmas. Typically, the iterative
application of these original regularity lemma is of essentially the same
strength as Szemer\'edi's regularity lemma as stated above and gives
similar tower-type bounds. However, for Szemer\'edi's proof of the Ramsey-Tur\'an result, only two iterations were needed, leading to a double-exponential bound.}  is a celebrated result of Szemer\'edi from
1972 in Ramsey-Tur\'an theory; see Theorem \ref{szemramtur}. For a graph
$H$ and positive integers $n$ and $m$, the Ramsey-Tur\'an number ${\bf
RT}(n,H,m)$ is the maximum number of edges a graph $G$ on $n$ vertices with
independence number less than $m$ can have without containing $H$ as a
subgraph. The study of Ramsey-Tur\'an numbers was introduced by S\'os \cite{S69}. It was motivated by the classical theorems of Ramsey and Tur\'an and their connections to geometry, analysis, and number theory. Ramsey-Tur\'an theory has attracted a great deal of
attention over the last 40 years; see the nice survey by Simonovits and
S\'os \cite{SS01}.

\begin{theorem}[Szemer\'edi \cite{Sz-orig}] \label{szemramtur}
  For every $\epsilon > 0$, there is a $\delta
  > 0$ for which every $n$-vertex graph with at least $\big( \frac{1}{8} +
  \epsilon \big) n^2$ edges contains either a $K_4$ or an independent set
  larger than $\delta n$.
  \label{thm:full-reg}
\end{theorem}

Four years later, Bollob\'as and Erd\H{o}s \cite{BoEr-RT-construct} gave a surprising geometric construction, utilizing the isoperimetric inequality for the high dimensional sphere, of a $K_4$-free graph on $n$ vertices with independence number $o(n)$ and $(\frac18-o(1))n^2$ edges. Roughly speaking, the Bollob\'as-Erd\H{o}s graph consists of two disjoint copies of a discretized Borsuk graph, which connect nearly antipodal points on a high dimensional sphere, with a dense bipartite graph in between which connects points between the two spheres which are close to each other. For details of this construction and its proof, see Section \ref{sectBEgraph}.

Bollob\'as and Erd\H{o}s asked to estimate the minimum possible
independence number in the critical window, when the number of edges is
about $n^2/8$. This remained one of the main open problems in this area,
and, despite considerable attention, not much progress has been made on
this problem. In particular, Bollob\'as and Erd\H{o}s asked the following
question.

\begin{problem}[From \cite{BoEr-RT-construct}] \label{toplower}
Is it true that for each $\eta>0$ there is an $\epsilon>0$ such that for
each $n$ sufficiently large there is a $K_4$-free graph with $n$ vertices,
independence number at most $\eta n$,  and at least
$(\frac{1}{8}+\epsilon)n^2$ edges?
\end{problem}

They also asked the following related problem, which was later featured in the Erd\H{o}s paper \cite{Erd90} from 1990 entitled ``Some of my favourite unsolved problems''.

\begin{problem}[From \cite{BoEr-RT-construct}] \label{middlelower}
Is it true that for every $n$, there is a $K_4$-free graph with $n$ vertices, independence number $o(n)$, and at least $\frac{n^2}{8}$ edges?
\end{problem}

Erd\H{o}s, Hajnal, Simonovits, S\'os, and Szemer\'edi \cite{EHSSS} noted that perhaps replacing $o(n)$ by a slightly smaller function, say by $\frac{n}{\log n}$, one could get smaller upper bounds on Ramsey-Tur\'an numbers. Specifically, they posed the following problem.

\begin{problem}[From \cite{EHSSS}] \label{problEHSSS}
Is it true  for some constant $c>0$ that ${\bf RT}(n,K_4,\frac{n}{\log n}) < (1/8-c)n^2$?
\end{problem}

This problem was further discussed in the survey by Simonovits and S\'os \cite{SS01} and by Sudakov \cite{Su-RT}. Motivated by this problem, Sudakov \cite{Su-RT} proved that if $m = e^{-\omega\left((\log n)^{1/2}\right)}n$, then ${\bf RT}(n,K_4,m) = o(n^2)$.

In this paper, we solve the Bollob\'as-Erd\H{o}s problem to estimate the minimum independence number in the critical window. In particular, we solve the above problems, giving positive answers to Problems  \ref{toplower} and \ref{middlelower}, and a negative answer to Problem \ref{problEHSSS}. We next discuss these results in depth.

The bound on $\delta$ as a function of $\epsilon$ in the now standard
proof of Theorem \ref{thm:full-reg} (sketched in
Section~\ref{sec:originalproof}) strongly depends on the number of parts in Szemer\'edi's regularity lemma. In particular, it shows that $\delta^{-1}$ can be taken to be a tower of twos of height $\epsilon^{-O(1)}$. However, the original proof of Szemer\'edi \cite{Sz-orig}, which used two applications of a weak regularity lemma, gives a better bound, showing that $\delta^{-1}$ can be taken to be double-exponential in $\epsilon^{-O(1)}$.

In the survey on the regularity method \cite{KSSS}, it is surmised that the some regularity lemma is likely
unavoidable for applications where the extremal graph has densities in the
regular partition bounded away from $0$ and $1$.  In particular, they
thought this should be the case for  Theorem \ref{thm:full-reg}. Contrary
to this philosophy, our first result is a new proof of Theorem
\ref{thm:full-reg} which gives a much better bound and completely avoids
using the regularity lemma or any notion similar to regularity. More
precisely, it gives a linear bound for $\delta$ on $\epsilon$ in Theorem
\ref{thm:full-reg}, in stark contrast to the double-exponential dependence
given by the original proof.

\begin{theorem}  \label{thm:no-reg}
  For every $\alpha$ and $n$, every $n$-vertex graph with at least
  $\frac{n^2}{8} + 10^{10} \alpha n$ edges contains either a copy of $K_4$
  or an independent set of size greater than $\alpha$.
\end{theorem}

It is natural to wonder whether one must incur a constant factor of
$10^{10}$.  Our second result sharpens the linear dependence down to a very
reasonable constant. Its proof uses the regularity lemma with an absolute
constant regularity parameter (independent of $n$ and $\alpha$).

\begin{theorem}   \label{thm:some-reg}
  There is an absolute positive constant $\gamma_0$ such that for every $\alpha <
  \gamma_0 n$, every $n$-vertex graph with at least $\frac{n^2}{8} +
  \frac{3}{2} \alpha n$ edges contains a copy of $K_4$ or an
  independent set of size greater than $\alpha$.
\end{theorem}

While Theorem \ref{thm:no-reg} has a weaker bound and a longer proof than Theorem \ref{thm:some-reg}, its inclusiion is justified  by the fact that its proof completely avoids using any regularity-like lemma, and the ideas may be of use to get rid of the regularity lemma in other applications. Further, it applies to all $\alpha$, while Theorem \ref{thm:some-reg} only applies to $\alpha<\gamma_0 n$.

We also prove the following corresponding lower bound, which shows that the
linear dependence in Theorem \ref{thm:some-reg} is best possible, matching
the dependence on $\alpha$ to within a factor of $3+o(1)$. Starting with
the Bollob\'as-Erd\H{o}s graph, the construction finds a slightly denser
$K_4$-free graph without increasing the independence number much. It also
gives a positive answer to Problem \ref{toplower} of Bollob\'as and
Erd\H{o}s with the linear dependence that our previous theorems now reveal
to be correct.  Here, we write $f(n) \ll g(n)$ when $f(n)/g(n) \rightarrow
0$ as $n \rightarrow \infty$.

\begin{theorem} \label{theorybeyond}
For
$\frac{(\log \log n)^{3/2}}{(\log n)^{1/2}} \cdot n
\ll
m
\leq
\frac{n}{3}$,
we have
\[
{\bf RT}(n,K_4,m) \geq
\frac{n^2}{8}+\left(\frac{1}{3}-o(1)\right)m n.
\]
\end{theorem}

\noindent \textbf{Remarks.} The tripartite Tur\'an graph has
independent sets of size $(1+o(1)) \frac{n}{3}$, so once $m$ exceeds
$\frac{n}{3}$, the Ramsey-Tur\'an problem for $K_4$ asymptotically coincides
with the ordinary Tur\'an problem.  Also, in the sublinear regime, our
proof actually produces the slightly stronger asymptotic lower bound ${\bf
RT}(n,K_4,m) \geq \frac{n^2}{8}+\left(\frac{1}{2}-o(1)\right)m n$ when
$\frac{(\log \log n)^{3/2}}{(\log n)^{1/2}} \cdot n \ll m \ll n$.

\bigskip

Bollob\'as and Erd\H{o}s drew attention to the interesting transition point of exactly $\frac{n^2}{8}$ edges. Thus far, the best result for this regime was a lower bound on the independence number of $n
e^{-O(\sqrt{\log n})}$ by Sudakov \cite{Su-RT}. The proof relies on a powerful probabilistic technique known as dependent random choice; see the survey by Fox and Sudakov \cite{FS-DRC}.

By introducing a new twist on the dependent random choice technique, we
substantially improve this lower bound on the independence number at the
critical point.  We think that this new variation may be interesting in
its own right, and perhaps could have other applications elsewhere, as the
main dependent random choice approach has now seen widespread use.  Our key
innovation is to exploit a very dense setting, and to select not the common
neighborhood of a random set, but the set of all vertices that have many
neighbors in a random set; then, we apply a ``dispersion'' bound on the
binomial distribution in addition to the standard Chernoff
``concentration'' bound.

\begin{theorem} \label{thm:1/8}
  There is an absolute positive constant $c$ such that every $n$-vertex graph with at least $\frac{n^2}{8}$ edges contains a copy of $K_4$ or an independent set of size greater than $c n
  \cdot \frac{\log \log n}{\log n}$.
\end{theorem}

We also prove an upper bound on this problem, giving a positive answer to
Problem \ref{middlelower} of Bollob\'as and Erd\H{o}s. The proof is again
by modifying the Bollob\'as-Erd\H{o}s graph to get a slightly denser
$K_4$-free graph whose independence number does not increase much.

\begin{theorem} \label{easythe}
There is an absolute positive constant $c'$ such that for each positive integer $n$, there is an $n$-vertex $K_4$-free graph with at least $\frac{n^2}{8}$ edges and independence number at most $c' n \cdot \frac{(\log \log n)^{3/2}}{(\log n)^{1/2}}$.
\end{theorem}

Recall that Bollob\'as and Erd\H{o}s \cite{BoEr-RT-construct} constructed a $K_4$-free graph on $n$ vertices with $(1-o(1))\frac{n^2}{8}$ edges with independence number $o(n)$. The various presentations of the proof of the Bollob\'as-Erd\H{o}s result in the literature \cite{BL}, \cite{BL2}, \cite{BoEr-RT-construct}, \cite{EHSSS}, \cite{EHSSS97}, \cite{SS01} do not give quantitative estimates on the little-$o$ terms. By finding good quantitative estimates for the relevant parameters, we can use the Bollob\'as-Erd\H{o}s graphs to prove the following theorem. This result gives a negative answer to Problem \ref{problEHSSS} of Erd\H{o}s, Hajnal, Simonovits, S\'os, and Szemer\'edi \cite{EHSSS}. It also complements the result of Sudakov \cite{Su-RT}, showing that the bound coming from the dependent random choice technique is close to optimal.

\begin{theorem}\label{corbeconstruct}
If $m=e^{-o\left((\log n / \log \log n)^{1/2}\right)}n$, then $${\bf RT}(n,K_4,m)
\geq \left(1/8-o(1)\right)n^2.$$
\end{theorem}

We summarize the results in the critical window in the following theorem. All of
the bounds, except for the first result in the first part, which is due to
Sudakov \cite{Su-RT}, are new.  As before, we write $f(n) \ll g(n)$ to
indicate that $f(n)/g(n) \rightarrow 0$ as $n \rightarrow \infty$.

\begin{theorem}\label{criticalwindow}
We have the following estimates. Here $c$, $c'$, and $\gamma_0$ are absolute constants.
\begin{enumerate}
\item If $m = e^{-\omega\left((\log n)^{1/2}\right)}n$, then ${\bf RT}(n,K_4,m) = o(n^2)$;  \\ \\
while if $m=e^{-o\left((\log n / \log \log n)^{1/2}\right)}n$, then ${\bf RT}(n,K_4,m)
\geq \left(1/8-o(1)\right)n^2$.
\item If $m = c n
  \cdot \frac{\log \log n}{\log n}$, then  ${\bf RT}(n,K_4,m)
\leq n^2/8$; \\ \\
while if $m=c' n \cdot \frac{(\log \log n)^{3/2}}{(\log n)^{1/2}}$, then ${\bf RT}(n,K_4,m)
\geq n^2/8$.

\item If $\frac{(\log \log n)^{3/2}}{(\log n)^{1/2}}\cdot n \ll m \leq
  \gamma_0 n$,
  we have
  \[
  \frac{n^2}{8}+\left(\frac{1}{3}-o(1)\right)m n \leq {\bf RT}(n,K_4,m)\leq
  \frac{n^2}{8}+\frac{3}{2}mn \,,
  \]
  where the constant $\frac{1}{3}$ can be replaced with $\frac{1}{2}$ in
  the range $m \ll n$.

\end{enumerate}

\end{theorem}

\medskip

\noindent {\bf Organization.}
In Section~\ref{sec:originalproof}, we recall the standard proof of Theorem \ref{thm:full-reg} using the regularity lemma. Our
new proof has two main steps. First, we show that every $K_4$-free
graph on $n$ vertices with at least $n^2/8$ edges and small
independence number must have a large cut, with very few non-crossing
edges. Second, we show that having a large cut implies the desired
Ramsey-Tur\'an result.

For the first step we present two different approaches. The first
approach, presented in Section~\ref{sec:great-cut-regularity}, is
conceptually simpler. Here we apply the regularity lemma with an
absolute constant level of precision and then apply the stability
result for triangle-free graphs to obtain a large cut that lets us
obtain Theorem~\ref{thm:some-reg}. The second approach, presented in
Section~\ref{sec:great-cut-no-regularity}, avoids using the regularity
lemma completely, and leads to Theorem~\ref{thm:no-reg}. Once we know
that the maximum cut is large, we proceed to the second step,
presented in Section~\ref{sec:stability}, where we obtain either a
$K_4$ or a large independent set. The conclusions of
the proofs are found in Section~\ref{sec:combine}. In Section~\ref{sec:DRC} we prove Theorem~\ref{thm:1/8} by introducing a new variant of the dependent random choice technique. In Section~\ref{sectBEgraph}, we give a quantitative proof of the Bollob\'as-Erd\H{o}s result, and use it to establish Theorem~\ref{corbeconstruct}. In Section~\ref{sectabove}, we show how to modify the Bollob\'as-Erd\H{o}s graph to get a slightly denser graph whose independence number is not much larger. We use this to establish Theorems \ref{theorybeyond} and \ref{easythe}. Finally, Section~\ref{sec:conclusion} contains some concluding remarks.
Throughout this paper, all logarithms are base $e$ unless otherwise
indicated. For the sake of clarity of presentation, we systematically omit
floor and ceiling signs whenever they are not crucial.

\section{The standard regularity proof} \label{sec:originalproof}

In this section we recall the standard proof of Theorem~\ref{thm:full-reg}. We reproduce the proof here because our proof of
Theorem~\ref{thm:some-reg} starts the same way.
We first need to properly state the regularity lemma, which requires some terminology. The edge density $d(X,Y)$ between two subsets of vertices of a graph $G$ is the fraction of pairs $(x, y) \in X \times Y$ that are edges of $G$. A pair $(X,Y)$ of vertex sets is called \emph{$\epsilon$-regular}
if for all $X' \subset X$ and $Y' \subset Y$ with $|X'| \geq \epsilon |X|$
and $|Y'| \geq \epsilon|Y|$, we have $|d(X',Y') - d(X,Y)| < \epsilon$. A
partition $V = V_1\cup \ldots \cup V_t$ is called equitable if $||V_i| -
|V_j|| \leq 1$ for all $i$ and $j$. The regularity lemma states that for
each $\epsilon>0$, there is a positive integer $M(\epsilon)$ such that the
vertices of any graph $G$ can be equitably partitioned $V(G) = V_1 \cup
\ldots \cup V_t$ into $\frac{1}{\epsilon} \leq t \leq M(\epsilon)$ parts where all but at most
$\epsilon t^2$ of the pairs $(V_i, V_j)$ are $\epsilon$-regular.

We next outline the standard proof of Theorem~\ref{thm:full-reg}. We apply Szemer\'edi's regularity lemma to
obtain a regular partition. The edge density between two parts cannot
exceed $\frac12 + \epsilon$, or else we can find a $K_4$ or a large
independent set. Then the reduced graph has density exceeding
$\frac12$, so by Mantel's theorem we can find three vertex sets
pairwise giving dense regular pairs, from which we can obtain a $K_4$
or a large independent set. We follow this outline with a few simple lemmas leading to the detailed proof.

\begin{lemma}
  Let $G$ be a $K_4$-free graph with independence number at most
  $\alpha$. Let $uv$ be an edge of $G$. Then $u$ and $v$ have at most $\alpha$
  common neighbors.
  \label{lem:edge-codegree}
\end{lemma}

\begin{proof} If we have an edge $uv$ whose endpoints have
codegree exceeding $\alpha$, then there is an edge $xy$ within the common
neighborhood of $u$ and $v$.  This forms a $K_4$.
\end{proof}

\begin{lemma}
  Let $t, \gamma > 0$ satisfy $\gamma t \leq 1$.  If $G$ is a $K_4$-free
  graph on $n$ vertices with independence number at most $\gamma n$, and $X$ and $Y$ are disjoint vertex subsets of size $n/t$, then the edge density between $X$ and $Y$ is at most
  $\frac{1}{2} + \gamma t$.
  \label{lem:pair-density-1/2}
\end{lemma}

\begin{proof}
  Let $A \subset X$ be the vertices with $Y$-degree greater than
  $\frac{n}{2t} + \frac{\gamma n}{2}$.  If $A$ contains an edge, then the
  endpoints of that edge will have neighborhoods in $Y$ that overlap in
  more than $\gamma n$ vertices, contradicting the $K_4$-freeness of $G$ by
  Lemma \ref{lem:edge-codegree}. Hence, $A$ is an independent set and
  $|A|/|X| \leq \gamma t$. It follows that the edge density between $X$ and
  $Y$ is at most
\begin{align*}
  \frac{|A|}{|X|} \cdot 1
  + \left( 1 - \frac{|A|}{|X|} \right) \cdot
  \frac{\frac{n}{2t} + \frac{\gamma n}{2}}{|Y|}
  &\leq
  (\gamma t) 1
  + (1 - \gamma t) \cdot \frac{\frac{n}{2t} + \frac{\gamma n}{2}}{n/t} \\
  &=
  \gamma t
  + (1 - \gamma t) \cdot \frac{1}{2} (1 + \gamma t) \\
  &<
  \frac{1}{2} + \gamma t \,. \qedhere
\end{align*}
\end{proof}

\begin{lemma}
  Suppose that $X$, $Y$, and $Z$ are disjoint subsets of size $m$, and each
  of the three pairs are $\epsilon$-regular with edge density at least
  $3\epsilon$.  Then there is either a $K_4$ or an independent set of size
  at least $4\epsilon^2 m$.
  \label{lem:regular-triangle}
\end{lemma}

\begin{proof}
We may assume that $\epsilon < \frac{1}{3}$, as
otherwise the given conditions are vacuous.  By the regularity condition, at most an $\epsilon$-fraction of
the vertices of $X$ fail to have $Y$-density at least $2\epsilon$, and
at most $\epsilon$-fraction fail to have $Z$-density at least
$2\epsilon$.  Select one of the other vertices $x \in X$, and let $Y'$ and
$Z'$ be $x$'s neighborhoods in $Y$ and $Z$.  At most an $\epsilon$-fraction of
the vertices of $Y$ fail to have $Z'$-density at least $2\epsilon$, so
among the vertices of $Y'$, there are still at least $\epsilon m$ of them
that have $Z'$-density at least $2\epsilon$.  Pick one such $y \in Y'$.
Now $x$ and $y$ have at least $(2\epsilon)^2 m$ common neighbors in $Z$,
and that is either an independent set, or it contains an edge $uv$ which
forms a $K_4$ together with $x$ and $y$.
\end{proof}

\medskip

Now we recall the standard proof of Theorem~\ref{thm:full-reg} using
the regularity lemma.

\begin{proof}[Proof of Theorem~\ref{thm:full-reg}]
  Suppose we have a $K_4$-free graph $G$ on $n$ vertices with at least
  $(\frac{1}{8}+\epsilon)n^2$ edges. Let $\beta=\epsilon/6$,
  $M=M(\beta)$ be the bound on the number of parts for Szemer\'edi's
  regularity lemma with regularity parameter $\beta$, and
  $\delta=\epsilon^2/(9M)$. So $M$ and $\delta^{-1}$ are at most a
  tower of height $\epsilon^{-O(1)}$. We apply Szemer\'edi's
  regularity lemma with regularity parameter $\beta$ to get a
  regularity partition into $\frac{1}{\beta} \leq t \leq M$ parts. For clarity of
  presentation, we ignore floor signs here and assume all parts have
  exactly $n/t$ vertices. At most $\beta t^2 (n/t)^2\leq
  \epsilon n^2/6$ edges
  go between pairs of parts which are not $\beta$-regular, and at most
  $\epsilon n^2/4$ edges go between parts which have edge density less than
  $\epsilon/2$ between them. The number of edges within individual parts is
  less than $t \cdot (n/t)^2/2 = n^2/(2t) \leq \beta n^2/2 = \epsilon
  n^2/12$.  Thus, more than
  $\left(\frac{1}{8}+\frac{\epsilon}{2}\right)n^2$ edges of $G$ go between
  pairs of parts which are $\beta$-regular and have edge density at least
  $\epsilon/2$ between them.

  Since $\epsilon > \delta t$, by Lemma~\ref{lem:pair-density-1/2}, if
  there is no independent set of size $\delta n$, then the edge density
  between each pair of parts in the regularity partition is less than
  $\frac{1}{2}+\epsilon$.  Consider the $t$-vertex reduced graph $R$ of the
  regularity partition, whose vertices are the parts of the regularity
  partition, and two parts are adjacent if the pair is $\beta$-regular and
  the edge density between them is at least $\epsilon/2$. As there are more
  than $\left(\frac{1}{8}+\frac{\epsilon}{2}\right)n^2$ edges between pairs
  of parts which form edges of $R$, and the edge density between each pair
  is less than $\frac{1}{2}+\epsilon$, the number of edges of $R$ is more
  than $\left(\frac{1}{8}+\frac{\epsilon}{2}\right)n^2
  \Big/\left[\left(\frac{1}{2}+\epsilon\right)(n/t)^2\right] >
  \left(\frac{1}{4} + \frac{\epsilon}{4}\right)t^2$. By Mantel's theorem,
  $R$ must contain a triangle. That is, the regularity partition has three
  parts each pair of which is $\beta$-regular and of density at least
  $\epsilon/2$. As $\beta = (\epsilon/2)/3$, Lemma~\ref{lem:regular-triangle}
  tells us that there is an independent set of size greater than
  $\frac{\epsilon^2}{9} \frac{n}{t} \geq \delta n$.
\end{proof}

\medskip

We note that the above proof can be modified to use the weak
regularity lemma by Frieze and Kannan \cite{FK96, FK99} to give a singly
exponential dependence between $\delta$ and $\epsilon$, i.e., $\delta
= 2^{-\poly(\epsilon^{-1})}$. This observation was made jointly with David
Conlon. Here is a rough sketch. We apply the weak regularity lemma
with parameter $\beta = \poly(\epsilon)$ to
obtain a weakly regular partition of the graph into $t \leq
2^{O(\beta^{-2})}$ parts, and let $\delta^{-1} = \poly(t/\epsilon)$. As before, no pair of parts can have density
exceeding $\frac{1}{2} + \epsilon$, so the reduced graph has at least
$\left(\frac14 + \frac{\epsilon}{16}\right) t^2$ edges. Using
Goodman's triangle supersaturation result~\cite{Goodman}, there are at
least $\Omega(\epsilon t^3)$ triangles in the reduced graph. Applying the
triangle counting lemma associated to the weak regular partition \cite{BCLSV}
(i.e., counting lemma with respect to the cut norm) we see that $G$
has at least $\Omega(\epsilon^4 n^3)$ triangles. We then conclude
as before to show that $G$ must contains a large
independent set.

\medskip

In each of the above proofs, we needed to apply a regularity lemma with
the input parameter depending on $\epsilon$, so the dependency of $\delta$ on
$\epsilon$ is at the mercy of the regularity lemma,
which cannot be substantially improved (see \cite{CoFo}). In the next
section, we start a new proof where we only need to apply the
regularity lemma with an absolute constant regularity parameter, so that we can
obtain a very reasonable linear dependence between $\delta$ and $\epsilon$. In
Section~\ref{sec:great-cut-no-regularity} we provide an alternate
approach which completely avoids the use of regularity.

\section{Large cut via regularity lemma}
\label{sec:great-cut-regularity}

Now we begin the proof of Theorem~\ref{thm:some-reg}. It is conceptually
easier than the regularity-free approach (Theorem \ref{thm:no-reg}), so we
start with it. The proof follows the same initial lines as the original
argument in Section~\ref{sec:originalproof}, except that we only use as
much regularity as we need to find a large cut.  Importantly, the cut is
deemed satisfactory once its size is within an absolute constant
approximation factor of the true maximum cut, which asymptotically contains
$1 - o(1)$ proportion of all of the graph's edges.  We only need to apply
the regularity lemma with a prescribed absolute constant level of
precision, and this is key to developing the sharper dependence on the
independence number.

We need the following stability version of Mantel's theorem to obtain
our large cut.

\begin{theorem}
  [Erd\H{o}s~\cite{Erd67}, Simonovits~\cite{Sim68}] For every $\epsilon > 0$,
  there is a $\delta > 0$ such that every $n$-vertex triangle-free graph
  with more than $\big( \frac{1}{4} - \delta \big) n^2$ edges is within
  edit distance $\epsilon n^2$ from a complete bipartite graph.
  \label{thm:ES-stability}
\end{theorem}

We use this result to obtain the following lemma.

\begin{lemma}
  For every $c > 0$ there is a $\gamma > 0$ such that every $K_4$-free graph $G$ on $n$ vertices with at least $\frac{n^2}{8}$
  edges and independence number less than $\gamma n$ has a cut which has at most $c n^2$ non-crossing edges.
  \label{lem:reg:great-cut}
\end{lemma}

\begin{proof}
  Let $\nu$ be the $\delta$ produced by Theorem
\ref{thm:ES-stability} when applied with $\frac{c}{2}$ as the input.  Let
$\epsilon = \min\big\{ \frac{\nu}{7}, \frac{c}{6} \big\}$.  We apply
Szemer\'edi's regularity lemma to $G$ with parameter $\epsilon$, to find a
partition of the vertex set into $t$ parts of equal size, where all but at
most $\epsilon t^2$ pairs of parts are $\epsilon$-regular, and
$\frac{1}{\epsilon} \leq t \leq M$.  Importantly, $M$ depends only on
$\epsilon$, and is completely independent of $n$.  Let $\gamma = 4
\epsilon^2 / M$.

Let $H$ be the reduced graph of the regularity partition. That is, $H$ is a graph on $t$ vertices where each vertex corresponds to one of the parts of the regularity partition.  Place an edge between a pair
of vertices in $H$ if and only if the corresponding pair of parts is
$\epsilon$-regular with edge density greater than $3\epsilon$.  The total
number of edges of $G$ not represented in $H$ is at most
\begin{align}
  \nonumber
  t \binom{n/t}{2}
  +
  \epsilon t^2 \left( \frac{n}{t} \right)^2
  +
  \binom{t}{2} (3\epsilon) \left( \frac{n}{t} \right)^2
  &<
  \frac{n^2}{2t}
  +
  \epsilon n^2
  +
  \frac{3}{2} \epsilon n^2 \\
  &\leq
  3\epsilon n^2 \,.
  \label{eq:reg:outside-H}
\end{align}
The first term came from the edges within individual parts of the
regularity partition, the second term came from pairs that were not
$\epsilon$-regular, and the third term came from pairs that had density at
most $3\epsilon$.

Let $m$ be the number of edges of $H$.  Lemma \ref{lem:pair-density-1/2}
bounds all pairwise densities by at most
\begin{displaymath}
  \frac{1}{2} + \gamma t
  \leq
  \frac{1}{2} + \gamma M
  =
  \frac{1}{2} + 4 \epsilon^2 \,.
\end{displaymath}
Therefore, the number of edges in the original graph is at most
\begin{displaymath}
  e(G)
  \leq
  m \left( \frac{1}{2} + 4 \epsilon^2 \right) \left( \frac{n}{t} \right)^2
  +
  3\epsilon n^2 \,.
\end{displaymath}
Yet we assumed that $G$ had at least $\frac{n^2}{8}$ edges, so dividing, we
find that
\begin{align*}
  m
  &\geq
  \frac{ \frac{1}{8} - 3 \epsilon }{ \frac{1}{2} + 4\epsilon^2 } \cdot t^2
  \\
  &>
  \left( \frac{1}{4} - 6 \epsilon \right)
  \left( 1 - 8\epsilon^2 \right) t^2 \\
  &>
  \left( \frac{1}{4} - 7 \epsilon \right) t^2 \,.
\end{align*}

As the independence number of $G$ is less than $\gamma n \leq 4\epsilon^2 n/M$ and the parts of the regularity partition have order $n/M$, by Lemma \ref{lem:regular-triangle} the auxiliary graph $H$ must be triangle-free. We may now appeal to the Erd\H{o}s-Simonovits stability (Theorem
\ref{thm:ES-stability}), which by our choice of $\epsilon$ implies that
$H$ is within $\frac{ct^2}{2}$ edges of being complete bipartite.  In
particular, there is a cut of $H$ which has at most $\frac{ct^2}{2}$
non-crossing edges.  Consider the corresponding cut of $G$.  Even if those
non-crossing edges of $H$ corresponded to pairs of full density, after
adding \eqref{eq:reg:outside-H} we find that the total number of
non-crossing edges of $G$ is at most
\begin{displaymath}
  \frac{ct^2}{2} \left( \frac{n}{t} \right)^2
  + 3 \epsilon n^2
  \leq
  cn^2 \,,
\end{displaymath}
as desired.
\end{proof}

\medskip

Now we deviate from the original regularity-based approach.  Our next
ingredient is a minimum-degree condition.

\begin{lemma}
  Let $G$ be a graph with $n$ vertices and $m$ edges, and suppose there is
  a vertex with degree at most $\frac{m}{n}$.  Delete the vertex from $G$,
  and let the resulting graph have $n' = n-1$ vertices and $m'$ edges.
  Then $\frac{m'}{n'} \geq \frac{m}{n}$.
  \label{lem:delete-density+}
\end{lemma}

\begin{proof}After deletion, the number of edges is $m' \geq m
- \frac{m}{n} = m \big( 1 - \frac{1}{n} \big) = m \cdot \frac{n-1}{n} = m
\cdot \frac{n'}{n}$, and therefore $\frac{m'}{n'} \geq \frac{m}{n}$.
\end{proof}

\begin{lemma}
Let $G$ be an $n$-vertex graph with at least $m$ edges. Then $G$ contains an induced subgraph $G'$ with $n'>2m/n$ vertices, at least $n'\frac{m}{n}$ edges, and minimum degree at least $\frac{m}{n}$.
  \label{lem:min-degree-1/8}
\end{lemma}

\begin{proof} Repeat the following procedure: as long as the
graph contains a vertex $v$ of degree at most $\frac{m}{n}$, remove $v$.  Let $G'$ be the resulting induced subgraph when this process terminates, and $n'$ be the number of vertices of $G'$. Note that at the very
beginning, the ratio of edges to vertices is at least $\frac{m}{n}$, and
by Lemma \ref{lem:delete-density+}, this ratio does not decrease in each iteration.  Therefore, throughout the process, the ratio of the number of edges to the number of vertices is always at least
$\frac{m}{n}$. Yet this ratio is precisely half of the average degree of the graph, which is less than the number of vertices of the graph, so we must have $n' > 2\frac{m}{n}$.  Also,  the number of edges of $G'$ is at least $m-(n-n')\frac{m}{n}=n'\frac{m}{n}$. Finally, as no more vertices are deleted, $G'$ has minimum degree at least $\frac{m}{n}$.
\end{proof}

\medskip

At this point, we switch gears, and introduce our regularity-free approach,
which will also reach this same point.  After both approaches have arrived
here, we will complete both proofs with the same argument.

\section{Large cut without regularity}
\label{sec:great-cut-no-regularity}

In this section, we assume the conditions of Theorem \ref{thm:no-reg}.

\begin{lemma}
  Theorem \ref{thm:no-reg} is trivial unless $\alpha \leq n / (2 \cdot
  10^{10})$.
  \label{lem:indep-cap}
\end{lemma}

\begin{proof} Theorem \ref{thm:no-reg} assumes that the number
of edges is at least $\frac{n^2}{8} + 10^{10} \alpha n$.  But if $\alpha >
n / (2 \cdot 10^{10})$, then this number already rises above
$\binom{n}{2}$, and the theorem becomes vacuous because there are no graphs
with that many edges.
\end{proof}

\begin{lemma}
  When we are proving Theorem \ref{thm:no-reg}, we may assume that all
  degrees are at least $\frac{n}{4} + (10^{10}-1)\alpha$, or else we are
  already done.
  \label{lem:min-degree}
\end{lemma}

\begin{proof}
  Let $C = 10^{10}$, so that we are proving that every graph with no $K_4$
  and no independent set of size greater than $\alpha$ must contain fewer
  than $\frac{n^2}{8} + C \alpha n$ edges.  We proceed by induction on $n$.
  Theorem \ref{thm:no-reg} is trivial unless $\alpha \geq 1$, in which case
  $C \alpha n$ is already at least $Cn$.  This exceeds $\binom{n}{2}$ for
  all $n \leq 2C$, so those serve as our base cases.

For the induction step, let $G$ be a graph with at least $\frac{n^2}{8} + C
\alpha n$ edges, and assume that the result is known for $n-1$.  Suppose
for the sake of contradiction that $G$ has no $K_4$ or independent set of
size greater than $\alpha$.  Let $\delta$ be its minimum degree, and delete
its minimum degree vertex.  The resulting graph also has no $K_4$ or independent set of size greater than $\alpha$, so by the induction
hypothesis,
\begin{displaymath}
  e(G) - \delta
  <
  \frac{(n-1)^2}{8} + C \alpha (n-1) \,.
\end{displaymath}
Yet we assumed that $e(G) \geq \frac{n^2}{8} + C \alpha n$.  Combining
these, we find that
\begin{align*}
  \frac{n^2}{8} + C \alpha n - \delta
  &<
  \frac{(n-1)^2}{8} + C \alpha (n-1) \\
  C \alpha + \frac{n}{4} - \frac{1}{8}
  &<
  \delta \,.
\end{align*}
Therefore, $\delta > \frac{n}{4} + (C-1) \alpha$.
\end{proof}

\medskip

This strong minimum degree condition establishes that every neighborhood
has size greater than $n/4$.  The first step of our regularity-free
approach associates a large set of neighbors to each vertex.

\begin{definition}
  For each vertex $v$ in $G$, arbitrarily select a set of exactly
  $n/4$ neighbors of $v$, and call that set $N_v$.  Define the
  remainder $R_v$ to be the complement of $N_v$.
  \label{def:Nv}
\end{definition}

\begin{definition}
  If a vertex $u \in R_v$ has density to $N_v$ in the range $[0.3, 0.34]$
  we say that $u$ \textbf{trisects} $v$.
\end{definition}

The next lemma blocks an extreme case which would otherwise obstruct our
proof.

\begin{lemma}
  Let $G$ be a graph on $n$ vertices with minimum degree at least
  $n/4$. Suppose that for every vertex $v$, all but at most $0.03n$ vertices
  of $R_v$ trisect $v$. Then there is either a $K_4$ or an independent set
  of size at least $n/1200$.
  \label{lem:trisect}
\end{lemma}

\begin{proof} Assume for the sake of contradiction that $G$ is $K_4$-free and the
  maximum independent set in $G$ has size $\alpha < n/1200$. As the minimum degree is at least $n/4$ and $G$ does not contain an independent set of this size, it must contain a triangle.
 Let $abc$ be an arbitrary triangle in the graph.
Define the three disjoint sets
\begin{align*}
  N_a^* &= N_a \setminus (N_b \cup N_c) \,, \\
  N_b^* &= N_b \setminus (N_a \cup N_c) \,, \\
  N_c^* &= N_c \setminus (N_a \cup N_b) \,.
\end{align*}
Let $m = n/4$.  Each of $N_a$, $N_b$, and $N_c$ has size exactly
$m$, and Lemma \ref{lem:edge-codegree} ensures that their pairwise
intersections are at most $\alpha$.  So, each of $N_a^*$, $N_b^*$, and
$N_c^*$ has size at least $m - 2\alpha$.  At most $0.06n$
vertices of $N_a^*$ fail to trisect either $b$ or $c$, so we may choose
$v \in N_a^*$ which trisects both $b$ and $c$.

Since we selected $v \in N_a^*$, it is adjacent to $a$, and therefore Lemma
\ref{lem:edge-codegree} implies that $v$ has at most $\alpha$ neighbors in
$N_a^*$.  By above, there are still at least $m - 3\alpha$ non-neighbors of
$v$ in $N_a^*$, of which at most $0.09n$ fail to trisect any of $v$, $b$, or
$c$.  Therefore, we may now select $u \in N_a^*$ which is non-adjacent to
$v$, and trisects each of $v$, $b$, and $c$.

Let $B = N_v \cap N_b^*$ and $C = N_v \cap N_c^*$.  We will establish two
claims: first, that $N_u$ intersects $B$ in more than $0.17 m$
vertices, and second, that $N_u$ intersects $C$ in more than $0.17 m$
vertices.  This is a contradiction, because $B$ and $C$ are disjoint
subsets of $N_v$, and the condition that $u$ trisects $v$ forces $|N_u \cap
N_v| \leq 0.34 m$.  By symmetry between $b$ and $c$, it suffices to prove only the first claim.

For this, suppose for the sake of contradiction that $N_u$ intersects $B$
in at most $0.17 m$ vertices.  Since $u$ trisects $b$, $u$ has at least
$0.3 m$ neighbors in $N_b$, hence at least $0.3 m - 2\alpha$ neighbors in
$N_b^*$, hence at least $0.13 m - 2\alpha > 0.03n + \alpha$ neighbors in
$N_b^* \setminus N_v$.  (Here, we used $\alpha < \frac{n}{1200}$.)  Of
these, at most $0.03 n$ fail to trisect $v$, and since the resulting number
is more than $\alpha$, there is an edge $xy$ such that $x, y \in N_b^*
\setminus N_v$, they both trisect $v$, and they both are adjacent to $u$.

Since $x$ and $y$ are adjacent, $N_x$ and $N_y$ overlap in at most $\alpha$
vertices by Lemma \ref{lem:edge-codegree}.  Since they both trisect $v$, we
conclude that $(N_x \cup N_y) \cap N_v$ has size at least $0.6 m -
\alpha$.  Also by Lemma \ref{lem:edge-codegree}, all but at most $2 \alpha$
of these vertices lie outside $N_b^*$, because $x$ and $y$ are adjacent to
$b$.  Thus, we have already identified at least $0.6 m - 3\alpha$
vertices of $N_v \setminus B$ that are adjacent to $x$ or $y$.  Yet $u$ is
adjacent to both $x$ and $y$, so by Lemma \ref{lem:edge-codegree}, $N_u$
can only include up to $2\alpha$ of these vertices.  Hence
\begin{align*}
  |N_u \cap (N_v \setminus B)|
  &\leq
  |N_v \setminus B| - (0.6 m - 3\alpha) + 2\alpha \\
  &=
  (m - |B|) - (0.6 m - 3\alpha) + 2\alpha \,.
\end{align*}
Since $v$ trisects $b$, we must have $B = N_v \cap N_b^*$ of size at least
$0.3 m - 2\alpha$.  Thus,
\begin{align*}
  |N_u \cap (N_v \setminus B)|
  &\leq
  (0.7 m + 2\alpha) - (0.6 m - 3\alpha) + 2\alpha \\
  &=
  0.1 m + 7\alpha \,.
\end{align*}
Since $u$ trisects $v$, we must have $|N_u \cap N_v| \geq 0.3 m$.
Therefore, $|N_u \cap B| \geq 0.2 m - 7\alpha$, which exceeds $0.17
m$ because $\alpha < n/1200$.  This establishes the claim, and
completes the proof of this lemma.
\end{proof}

The next lemma is a simple averaging argument which will be useful in the
lemma that follows.

\begin{lemma}
  Let $a_1, \ldots, a_m$ be a sequence of real numbers from $[0, 1]$ whose
  average exceeds $1/3$.  Suppose that at most 0.1\% of them
  exceed $0.3334$.  Then at most 3\% of them lie outside the range
  $[0.3, 0.34]$.
  \label{lem:max-bounds-min}
\end{lemma}

\begin{proof} On the contrary, if at least 2.9\% of them fall
below 0.3, then the average of the sequence is at most
\begin{displaymath}
  0.029 \cdot 0.3 + 0.97 \cdot 0.3334 + 0.001 \cdot 1
  =
  0.333098
  <
  \frac{1}{3} \,,
\end{displaymath}
because the maximum value is at most 1.
\end{proof}

Using the preceding two lemmas, we deduce the next lemma, which shows that if the independence number is small in a $K_4$-free graph, then there is a vertex $v$ such that a substantial fraction of the vertices in $R_v$ have density substantially larger than $1/3$ to $N_v$.

\begin{lemma} Suppose that $\alpha < n/1200$.
  In every $K_4$-free graph on $n$ vertices with independence number
  at most $\alpha$ and minimum degree greater than $\frac{n}{4} +
  \alpha$, there exists a vertex $v$ for which over $0.1\%$ of the
  vertices of $R_v$ have density greater than $0.3334$ to $N_v$.
  \label{lem:Rv-uneven}
\end{lemma}

\begin{proof} For every vertex $v$, every $u \in N_v$ has at
most $\alpha$ neighbors in $N_v$ by Lemma \ref{lem:edge-codegree}. Then $u$ has more than $n/4$
neighbors in $R_v$.  In particular, the density of the
bipartite subgraph between $N_v$ and $R_v$ is strictly greater than $1/3$.
Therefore, by Lemma \ref{lem:max-bounds-min}, each vertex $v$ which fails
the property produces a situation where all but 3\% of the vertices of
$R_v$ trisect $v$.  If this occurs for every vertex
$v$, then we satisfy the main condition of Lemma \ref{lem:trisect}. \end{proof}

\begin{lemma}
  For any $0< c \leq 3/4$, the following holds with $C = \frac{9}{8c} + \frac{1}{2}$.  Let $G$ be an $n$-vertex graph with no $K_4$, in which all
  independent sets have size at most $\alpha$, and suppose that $\alpha
  \leq cn/3$.  Let $R$ be a subset of $3n/4$ vertices, and
  let $T \subset R$ have size $cn$.  Suppose that every vertex of $T$ has
  degree (in $R$) at least $\frac{n}{8} + C \alpha$.  Then there is a
  subset $U \subset R$ (not necessarily disjoint from $T$) of size at least
  $n/4$ such that every vertex of $U$ has more than $\alpha$
  neighbors in $T$.
  \label{lem:R-up-down}
\end{lemma}

\begin{proof} Greedily pull out a matching from $G[T]$ of
$\frac{cn}{3}$ edges.  This is possible because $G[T]$ has independence
number at most $\alpha \leq \frac{cn}{3}$.  Create an auxiliary bipartite
graph $H$ with two sides $A$ and $B$ as follows.  Set $|A| = \frac{cn}{3}$,
with one vertex for each of the matching edges.  Let $B$ be a copy of $R$.
Place an edge between $a \in A$ and $b \in B = R$ whenever the vertex $b
\in R$ is adjacent to at least one of the endpoints of the matching edge
corresponding to $a$.  Since every vertex in $T$ has degree in $R$ at least $\frac{n}{8} +
C \alpha$, and no independent set of size larger than $\alpha$, Lemma
\ref{lem:edge-codegree} implies that in the auxiliary bipartite graph $H$,
every vertex of $A$ has degree at least $\frac{n}{4} + (2C-1)\alpha$.

Let $U$ contain all vertices of $B$ that have degree (in $H$) greater
than $\alpha$. Since the sum of all degrees of $B$ equals the sum of all
degrees of $A$, this sum is at least $\big( \frac{cn}{3} \big) \big(
\frac{n}{4} + (2C-1)\alpha \big)$.  At the same time, it
is also at most $|R \setminus U| \alpha + |U| \big( \frac{cn}{3} \big)$.  Putting
these together, we find that
\[
  (|R| - |U|) \alpha + |U| \left( \frac{cn}{3} \right)
  \geq
  \left( \frac{cn}{3} \right) \left( \frac{n}{4} + (2C-1)\alpha
  \right)
\]
so that
\begin{align*}
|U| &\geq
  \frac{
  \left( \frac{cn}{3} \right) \left( \frac{n}{4} + (2C-1)\alpha \right)
  - \left( \frac{3n}{4} \right) \alpha }
  { \frac{cn}{3} - \alpha } \\
  &>
  \frac{
  \left( \frac{cn}{3} \right) \left( \frac{n}{4} + (2C-1)\alpha \right)
  - \left( \frac{3n}{4} \right) \alpha }
  { \frac{cn}{3} } \\
  &= \frac{n}{4} + (2C-1)\alpha - \frac{9}{4c} \alpha
  \\
  & = \frac{n}{4}.
\end{align*}
Finally, from the definition of $U$, it follows that (in $G$) every vertex in $U$ has at least $\alpha$ neighbors in $T$.
\end{proof}

\begin{lemma}
  In a graph, let $L$ be a subset of vertices, and let $xyz$ be a
  triangle.  (The vertices $x$, $y$, and $z$ each may or may not lie
  in $L$.) Suppose that the $L$-degrees of $x$, $y$, and $z$ sum up to
  more than $|L| + 3\alpha$.  Then the graph contains a $K_4$ or an
  independent set of size greater than $\alpha$.
  \label{lem:L-triangle}
\end{lemma}

\begin{proof} Let $X$, $Y$, and $Z$ be the neighborhoods of
$x$, $y$, and $z$ within $L$.  By inclusion-exclusion,
\begin{align*}
  |L|
  \geq
  |X \cup Y \cup Z|
  &\geq
  (|X| + |Y| + |Z|) - |X \cap Y| - |Y \cap Z| - |Z \cap X| \\
  &>
  (|L| + 3\alpha) - |X \cap Y| - |Y \cap Z| - |Z \cap X| \,.
\end{align*}
Thus at least one of the pairwise intersections between $X$, $Y$, and $Z$
exceeds $\alpha$; without loss of generality, suppose it is between the
$L$-neighborhoods of $x$ and $y$.  If this intersection is an independent
set, then we have found an independent set of size greater than $\alpha$.
Otherwise, it spans an edge $uv$, and $xyuv$ forms a copy of $K_4$.
\end{proof}

\begin{corollary}
  In a $K_4$-graph with independence number at most $\alpha$, let
  $L$ and $X$ be disjoint subsets of vertices.  Suppose that every
  vertex of $X$ has $L$-degree greater than $\frac{|L|}{3} + \alpha.$
  Then the induced subgraph on $X$ has maximum degree at most
  $\alpha$.
  \label{cor:1/3-sparse}
\end{corollary}

\begin{proof} Suppose a vertex $x \in X$ has more than $\alpha$
neighbors in $X$.  This neighborhood cannot be an independent set, so it
spans an edge $yz$.  Now $xyz$ is a triangle whose vertices have $L$-degree
sum greater than $|L| + 3\alpha$, and Lemma \ref{lem:L-triangle} completes
the proof.  \end{proof}

Our next lemma establishes a major milestone toward constructing a cut
which contains almost all of the edges.  Such a bipartition spans few edges
within each part, and the following lemma achieves this for one part.

\begin{lemma}
  For any $0 < c < \frac{1}{2}$, the following holds with $C =
  \frac{9 \cdot 10^5}{8c} + 1$. In a $K_4$-free graph on $n$ vertices with minimum degree
  $\frac{n}{4} + C \alpha$ and independence number at most $\alpha
  \leq 10^{-5} \cdot \frac{cn}{3}$, there must exist a subset $X$ of $\left( \frac{1}{2} - c \right)n$
  vertices for which its induced subgraph has maximum degree at most
  $\alpha$.
  \label{lem:big-sparse}
\end{lemma}

\begin{proof} Use Lemma \ref{lem:Rv-uneven} to select a vertex
$v$ for which over 0.1\% of the vertices of $R_v$ have density greater than
0.3334 to $N_v$.  Let $L = N_v$ and let $R = R_v$.  Let $c_1 = 10^{-5}
c$, so that $C - \frac{1}{2}$ is the constant obtained from Lemma \ref{lem:R-up-down} with
parameter $c_1$.

Each vertex $u \in L$ has degree at least $\frac{n}{4} + C \alpha$ by
assumption, but by Lemma \ref{lem:edge-codegree}, since $u$ is adjacent to
$v$, at most $\alpha$ of this degree can go back to $L$.  Therefore, every
vertex of $L$ has more than $|R|/3$ neighbors in $R$, which implies that
the total number of edges between $L$ and $R$ exceeds $|L| |R|/3$.

Let $A \subset R$ be the vertices of $R$ whose $L$-degree exceeds
$\frac{|L| + \alpha}{2}$.  If $|A| > \alpha$, then $A$ cannot be an
independent set, so it induces an edge $wx$; each endpoint has $L$-degree
greater than $\frac{|L| + \alpha}{2}$, so $w$ and $x$ have more than
$\alpha$ common neighbors in $L$.  That common neighborhood is too large to
be an independent set, so it must induce an edge $yz$, and $wxyz$ forms a
copy of $K_4$.  Therefore, $|A| \leq \alpha$.

Every vertex of $R \setminus A$ has total degree at least $\frac{n}{4} + C
\alpha$ by assumption, at most $\frac{|L| + \alpha}{2}$ of which goes to
$L$ by construction.  At least $\frac{n}{8} + \big( C - \frac{1}{2}
\big) \alpha$ remains within $R$.  Let $T$ be the $c_1 n$ vertices of $R
\setminus A$ of highest $L$-degree.  We may now apply Lemma
\ref{lem:R-up-down} on $R$ and $T$, and find $U \subset R$ of size exactly
$\frac{n}{4} = \frac{|R|}{3}$, each of whose vertices has more than
$\alpha$ neighbors in $T$.

Let $x$ be the vertex in $U$ with highest $L$-degree, and let its
$L$-degree be $a$.  Let $b$ be the smallest $L$-degree of a vertex in $T$.
Since $|T \cup A| \leq c_1 n + \alpha < 0.1\% |R|$, we must have
\begin{equation}
  b > 0.3334 |L| \,.
  \label{eq:b>1/3}
\end{equation}
by the initial choice of $v$.  Since $x$ has more than $\alpha$ neighbors
in $T$, its neighborhood in $T$ spans an edge $yz$, forming a triangle
$xyz$.  The sum of the $L$-degrees of its vertices is at least $a + 2b$.
If this exceeds $|L| + 3\alpha$, then we are already done by Lemma
\ref{lem:L-triangle}, so we may now assume that
\begin{equation}
  a + 2b \leq |L| + 3\alpha \,.
  \label{eq:a+2b<1}
\end{equation}
To put $\alpha$ in perspective, note that our initial assumption on
$\alpha$ translates into
\begin{equation}
  \alpha
  \leq 10^{-5} \cdot \frac{cn}{3}
  = \frac{c_1 n}{3}
  = \frac{4}{3} \cdot c_1 |L| \,.
  \label{eq:alpha<<L}
\end{equation}

If $U$ and $T$ overlap at all, then we also have $a \geq b$, so inequality
\eqref{eq:b>1/3} then forces both $a, b > 0.3334 |L|$.  Combining this with
inequality \eqref{eq:a+2b<1}, we find that $0.0002 |L| < 3 \alpha$, and since
$|L| = \frac{n}{4}$, we have $\frac{n}{60000} < \alpha$.  This is
impossible, because we assumed that $\alpha < 10^{-5} \cdot \frac{cn}{3}$,
and $c < \frac{1}{2}$.  Therefore, $U$ and $T$ are disjoint, and we may
upper bound the sum of all $L$-degrees from $R$ by
\begin{equation}
  e(L, R)
  \leq
  |U| a
  + (|R| - |U|) b
  + |T \cup A| (|L| - b) \,.
  \label{eq:eLR-upper}
\end{equation}
This is because all vertices of $U$ have $L$-degree at most $a$, and
of the remaining vertices of $R$, only those in $T \cup A$ may have
$L$-degree exceeding $b$; even then, all $L$-degrees are at most $|L|$.
Simplifying this expression with $|U| = |R|/3$, $|T \cup A| \leq c_1 n +
\alpha$, and inequalities \eqref{eq:a+2b<1} and \eqref{eq:alpha<<L}, we
find that the total $L$-degree sum from $R$ is
\begin{align*}
  e(L, R)
  &\leq
  \left( \frac{|R|}{3} \right) a
  + \left( \frac{2 |R|}{3} \right) b
  + (c_1 n + \alpha) (|L| - b) \\
  &<
  \frac{|R|}{3} (a + 2b)
  + (c_1 n + \alpha) |L| \\
  &\leq
  \frac{|R|}{3} (|L| + 3\alpha)
  + (c_1 n + \alpha) |L| \\
  &=
  \frac{|L| |R|}{3} + \alpha (|L| + |R|) + c_1 n |L| \\
  &=
  \frac{|L| |R|}{3} + \alpha (n) + c_1 n |L| \\
  & \leq
  \frac{|L| |R|}{3} + \frac{7}{3} \cdot c_1 n |L| \,.
\end{align*}
Yet one of our first observations was that $e(L, R) > |L| |R| / 3$.
Therefore, the total amount of slack in inequality \eqref{eq:eLR-upper} is
at most $\frac{7}{3} \cdot c_1 n |L|$.

This is a very small gap.  To take advantage of it, let $S$ be the subset
of vertices in $R \setminus (A \cup U)$ whose $L$-degree is at most
$\frac{|L|}{3} + \alpha$, which is less than $b$ by inequalities \eqref{eq:alpha<<L} and
\eqref{eq:b>1/3}. Hence, $S$ is entirely contained in $R \setminus (A \cup T \cup U)$.  We may then sharpen
inequality \eqref{eq:eLR-upper} to
\begin{equation}
  e(L, R)
  \leq
  |U| a
  + (|R| - |U|) b
  + |T \cup A| (|L| - b)
  - |S| \left( b - \frac{|L|}{3} - \alpha \right)
  \,.
  \label{eq:eLR-upper2}
\end{equation}
In particular, the new summand cannot exceed the amount of slack we
previously determined, and so
\[
  |S| \left( b - \frac{|L|}{3} - \alpha \right)
  <
  \frac{7}{3} \cdot c_1 n |L|.
\]
Hence
\[
  |S|
  <
  \frac{7 c_1 n |L|}{3 b - |L| - 3\alpha}
 =\frac{7 c_1 n}{\frac{3 b}{|L|} - 1 - \frac{3\alpha}{|L|}}.
\]
Combining this with inequalities \eqref{eq:b>1/3} and \eqref{eq:alpha<<L},
we conclude that
\[
  |S|
  <
  \frac{7 c_1 n}{ 1.0002 - 1 - 4 c_1 }
  <
  \frac{7 c_1 n}{ 1.0002 - 1 - 2 \cdot 10^{-5} }
  =
  \frac{7 \cdot 10^{-5}}{ 1.0002 - 1 - 2 \cdot 10^{-5} } \cdot cn
  < 0.39 cn \,,
\]
and so if we define $X = R \setminus (U \cup S \cup A)$, the size of $X$ is
at least $\big( \frac{1}{2} - c \big)n$.  Furthermore, every vertex of $X$
has $L$-degree greater than $\frac{|L|}{3} + \alpha$, and so Corollary
\ref{cor:1/3-sparse} implies that the induced subgraph $G[X]$ has all
degrees at most $\alpha$.  \end{proof}

\begin{lemma}
  For any $0 < c^* < \frac{2}{5}$, the following holds with $c =
  c^*/4$.  Let $G$ be a $K_4$-free graph on $n$ vertices with
  independence number at most $\alpha < \frac{c^* n}{50}$ and minimum
  degree at least $\frac{n}{4}$.  Suppose it has a set $X$ of $\big(
  \frac{1}{2} - c \big)n$ vertices, which induces a subgraph of maximum
  degree at most $\alpha$.  Then in the max-cut of $G$, the total number of
  non-crossing edges is at most $c^* n^2$.
  \label{lem:half-sparse=big-cut}
\end{lemma}

\begin{proof}
  We will use $c < \frac{1}{10}$.  Let $Y$ be the complement of $X$.
  It suffices to show that the total number of edges spanned within
  each of $X$ and $Y$ is at most $c^* n^2$, because the max-cut can only
  do better.  Since $G[X]$ has maximum degree at most $\alpha$, we
  clearly have $e(X) \leq \frac{\alpha |X|}{2}$.

By the minimum degree condition, each vertex of $X$ must have degree at
least $\frac{n}{4}$, and at most $\alpha$ of its neighbors can fall back in
$X$.  Therefore, the total number of edges from $X$ to $Y$ is already
\begin{equation}
  e(X, Y)
  \geq
  \left( \frac{1}{2} - c \right)n \cdot \left( \frac{n}{4} - \alpha \right)
  >
  \frac{n^2}{8} - \frac{cn^2}{4} -  \frac{\alpha n}{2} \,.
  \label{eq:half-sparse-eXY-}
\end{equation}
Let $A \subset Y$ be the vertices of $Y$ which have more than
$\frac{|X| + \alpha}{2}$ neighbors in $X$.  As in the beginning of the
proof of Lemma \ref{lem:big-sparse}, we must have $|A| \leq \alpha$.
Summing the $X$-degrees of the vertices in $Y$, we find that
\begin{align}
  \nonumber
  e(X, Y)
  &\leq
  \alpha |X| + (|Y| - \alpha) \cdot \frac{|X| + \alpha}{2} \\
  \nonumber
  &=
  \alpha \left( \frac{1}{2} - c \right)n
  + \frac{1}{2} \left( \frac{n}{2} + cn - \alpha \right)
  \left( \frac{n}{2} - cn + \alpha \right) \\
  \nonumber
  &=
  \frac{\alpha n}{2}-c \alpha n
  + \frac{1}{2} \left(
  \frac{n^2}{4} - (cn - \alpha)^2 \right) \\
  &<
  \frac{\alpha n}{2} +
  \frac{n^2}{8}
  \label{eq:half-sparse-eXY+}
\end{align}
The amount of slack between the bounds for $e(X, Y)$ in
\eqref{eq:half-sparse-eXY-} and \eqref{eq:half-sparse-eXY+} is at most
$\alpha n + \frac{cn^2}{4}$.

Let $S$ be the subset of vertices in $Y$ whose $X$-degree is at most
$\frac{|X|}{3} + \alpha$.  Just as in the proof of Lemma
\ref{lem:big-sparse}, we may use our bound on the slack to control the size
of $S$.  Indeed, in our upper bound \eqref{eq:half-sparse-eXY+}, we used a
bound of at least $\frac{|X| + \alpha}{2}$ for every vertex of $Y$.  Each
vertex of $S$ now reduces the bound of \eqref{eq:half-sparse-eXY+} by
\begin{equation}
  \frac{|X|}{6} - \frac{\alpha}{2}
  \geq
  \frac{n}{15} - \frac{\alpha}{2}
  \geq
  \frac{n}{20} \,.
  \label{eq:half-sparse-1/3-boost}
\end{equation}
Here, we used $c < \frac{1}{10}$ to bound $|X| \leq 0.4 n$, and $\alpha
\leq \frac{n}{30}$.  Therefore, the size of $S$ is at most the slack
divided by \eqref{eq:half-sparse-1/3-boost}:
\begin{equation}
  |S|
  \leq
  \left( \alpha n + \frac{cn^2}{4} \right) / \left( \frac{n}{20} \right)
  =
  20 \alpha + 5cn \,.
  \label{eq:half-sparse-S}
\end{equation}
Using this, we may finally bound the number of edges in $Y$.  The key
observation is that Corollary \ref{cor:1/3-sparse} forces the induced
subgraph on $Y \setminus S$ to have maximum degree at most $\alpha$.
Therefore, even if $S$ were complete to itself and to the rest of $Y$,
\begin{displaymath}
  e(Y)
  \leq
  \frac{\alpha (|Y| - |S|)}{2}
  +
  \frac{|S|^2}{2}
  +
  |S| \cdot (|Y| - |S|)
  <
  |Y| |S| + \frac{\alpha |Y|}{2} \,.
\end{displaymath}
Combining this with \eqref{eq:half-sparse-S} and our initial bound on
$e(X)$, we obtain
\begin{align*}
  e(X) + e(Y)
  &<
  |Y| |S| + \frac{\alpha n}{2} \\
  &\leq
  (0.6n) (20 \alpha + 5cn) + \frac{\alpha n}{2} \\
  &=
  3cn^2 + 12.5 \alpha n \\
  &\leq 4cn^2 = c^* n^2 \,.
\end{align*}
Here, we used $c < \frac{1}{10}$ to bound $|Y| \leq 0.6n$, and $\alpha
\leq \frac{c^* n}{50} = \frac{cn}{12.5}$.  This completes the proof.
\end{proof}

\section{Refinement of stability}
\label{sec:stability}

Both arguments have now found very good cuts.  In this section, we
show how to finish the argument from this point.

\begin{lemma}
  Let $G$ be a $K_4$-free graph on $n$ vertices, at least $\frac{n^2}{8}$ edges,and
  independence number at most $\alpha \leq cn$.  Suppose its vertices have
  been partitioned into $L \cup R$, and $e(L) + e(R) \leq c n^2$.  Then
  $|L|$ and $|R|$ are both within the range $\big( \frac{1}{2} \pm
  \sqrt{3c} \big) n$.
  \label{lem:great-cut=1/2}
\end{lemma}

\begin{proof} Without loss of generality, suppose that $|L|
\leq |R|$, and let $|L| = \frac{n}{2} - l$.  The same argument
that yielded \eqref{eq:half-sparse-eXY+} implies that
\begin{align*}
  e(L, R)
  &\leq
  \alpha |L| + (|R| - \alpha) \cdot \frac{|L| + \alpha}{2} \\
  &=
  \alpha |L| + \frac{1}{2} \left( \frac{n}{2} + l - \alpha \right)
  \left( \frac{n}{2} - l + \alpha \right) \\
  &<
  \alpha |L| + \frac{n^2}{8} - \frac{l^2}{2} + l \alpha \\
  &=
  \frac{n^2}{8} - \frac{l^2}{2} + \frac{\alpha n}{2} \,.
\end{align*}
Combining this with the assumed lower bound on $e(G)$, assumed upper
bound on $e(L) + e(R)$, and $\alpha \leq cn$, we find that
$$ \frac{n^2}{8}
  \leq
  e(G)
  \leq
  cn^2 + \frac{n^2}{8} - \frac{l^2}{2} + \frac{\alpha n}{2},$$
and hence
 $$\frac{l^2}{2} \leq \frac{3cn^2}{2},$$ and   $l \leq \sqrt{3c} \cdot n$,
as desired.  \end{proof}

\medskip

The next result actually uses an extremely weak condition on the minimum
degree.  It leverages it by taking a max-cut, which has the nice property
that every vertex has at least as many neighbors across the cut as on its
own side.  This local optimality property immediately translates the
minimum degree condition to a minimum cross-degree condition, which is very
useful.  Although it may seem like we are re-using many of the techniques
that we introduced for earlier parts of this proof, we are not re-doing the
same work, because we are now proving properties for the max-cut, which
\emph{a priori}\/ could be somewhat different from the partitions obtained
thus far.

\begin{lemma}
  Let $G$ be a $K_4$-free graph on $n$ vertices with minimum degree at
  least $cn$ and independence number at most $\alpha \leq \frac{cn}{36}$.
  Let $L \cup R$ be a max-cut with $\frac{n}{3} \leq |R| \leq
  \frac{2n}{3}$.  Let
  $T \subset L$ be the vertices with $R$-degree greater than $\left(
  \frac{1}{2} - \frac{c}{8} \right) |R|$.  Then every vertex of $L$ has at
  most $\alpha$ neighbors in $T$.
  \label{lem:very-high}
\end{lemma}

\begin{proof}
  The minimum degree condition and the local optimality property of the
  max-cut implies that every vertex of $L$ has $R$-degree at least
  $\frac{cn}{2} > \frac{c |R|}{2}$.  Suppose that $L$ contains a
  triangle which has at least two vertices in $T$.  Then, the sum of the
  triangle's $R$-degrees would exceed
  \begin{displaymath}
    2\left(\frac{1}{2}-\frac{c}{8}\right)|R|+\frac{c|R|}{2}
    =
    |R| + \frac{c|R|}{4} \geq |R| + \frac{cn}{12} \geq |R| + 3\alpha \,.
  \end{displaymath}
  This is impossible by Lemma \ref{lem:L-triangle}.

Now, suppose for the sake of contradiction that some vertex $v \in L$ has
more than $\alpha$ neighbors in $T$.  This neighborhood is too large to be
an independent set, and therefore it contains an edge with both endpoints
in $T$.  That edge, together with $v$, forms one of the triangles
prohibited above.  \end{proof}

\begin{lemma}
  For any $0 < c < 1$, the following holds with $c' = c^2/800$. Assume
  $\alpha < cn/300$.
  Let $G$ be a $K_4$-free graph on $n$ vertices with
  at least $\frac{n^2}{8} + \frac{3 \alpha n}{2}$ edges, and minimum
  degree at least $cn$. Suppose that the max-cut of $G$ partitions
  the vertex set into $L \cup R$ such that $e(L) + e(R) \leq c' n^2$.
  Then $G$ either has a copy of $K_4$, or an independent set of size
  greater than $\alpha$.
  \label{lem:great-cut=done}
\end{lemma}

\begin{proof} Assume for the sake of
contradiction that $G$ has no $K_4$ or independent sets larger than
$\alpha$.  Let $A_L \subset L$ be the vertices whose $R$-degree exceeds
$\frac{|R| + \alpha}{2}$, and let $A_R \subset R$ be the vertices whose
$L$-degree exceeds $\frac{|L| + \alpha}{2}$.  As in the beginning of the
proof of Lemma \ref{lem:big-sparse}, we must have $|A_L|, |A_R| \leq
\alpha$.

Next, let $S_L \subset L$ be the vertices whose $R$-degree is at most
$\big( \frac{1}{2} - \frac{c}{8} \big) |R|$, and let $S_R \subset R$ be the
vertices whose $L$-degree is at most $\big( \frac{1}{2} - \frac{c}{8} \big)
|L|$.  We first show that $S_L$ and $S_R$ must be small.  For this, we
count crossing edges in two ways.  If we add all $R$-degrees of vertices in
$L$, and all $L$-degrees of vertices in $R$, then we obtain exactly $2 e(L,
R)$.  Since $|A_L|, |A_R| \leq \alpha$, we can bound this sum by
\begin{align*}
  2 e(L, R)
  &\leq
  \left[ \alpha |R| + |S_L| \left( \frac{1}{2} - \frac{c}{8} \right) |R|
  + (|L| - \alpha - |S_L|) \frac{|R| + \alpha}{2} \right] \\
  & \quad \quad +
  \left[ \alpha |L| + |S_R| \left( \frac{1}{2} - \frac{c}{8} \right) |L|
  + (|R| - \alpha - |S_R|) \frac{|L| + \alpha}{2} \right] \,.
\end{align*}
The first bracket simplifies to
\begin{displaymath}
  \frac{|L| |R|}{2} - |S_L| \left( \frac{c |R|}{8} + \frac{\alpha}{2} \right)
  + \alpha \left( |R| + \frac{|L|}{2} - \frac{|R|}{2} \right) -
  \frac{\alpha^2}{2} \, .
\end{displaymath}
Since $c' < \frac{1}{300}$ and $\alpha < \frac{n}{300}$, Lemma
\ref{lem:great-cut=1/2} bounds $|R| > 0.4 n$.  Therefore, the first bracket
is less than
\begin{displaymath}
  \frac{|L| |R|}{2} - |S_L| \left( \frac{cn}{20} \right)
  + \alpha \cdot \frac{|L| + |R|}{2}
  =
  \frac{|L| |R|}{2} - |S_L| \left( \frac{cn}{20} \right)
  + \frac{\alpha n}{2} \,,
\end{displaymath}
and similarly with the second bracket. Hence
\begin{equation}
  e(L, R)
  <
  \frac{|L| |R|}{2} - |S_L| \left( \frac{cn}{40} \right)
  - |S_R| \left( \frac{cn}{40} \right)
  + \frac{\alpha n}{2} \,.
  \label{eq:end:small-S:eLR+}
\end{equation}

On the other hand, we were given that $e(L) + e(R) \leq c' n^2$, while also
$e(G) \geq \frac{n^2}{8} + \frac{3 \alpha n}{2}$.  Therefore, we must also
have
\begin{equation}
  e(L, R) \geq \frac{n^2}{8} - c' n^2 + \frac{3 \alpha n}{2} \,.
  \label{eq:end:small-S:eLR-}
\end{equation}
Combining \eqref{eq:end:small-S:eLR-}, \eqref{eq:end:small-S:eLR+}, and $|L||R| \leq \frac{n^2}{4}$, we find that
\begin{equation}
   |S_L| \left( \frac{cn}{40} \right) + |S_R| \left( \frac{cn}{40} \right)
   <
   c' n^2 \,,
   \label{eq:end:small-S:combine}
\end{equation}
and in particular, both $|S_L|$ and $|S_R|$ are at most $\frac{40 c'
n}{c}$.  Since we defined $c' = \frac{c^2}{800}$, we have
\begin{equation}
  |S_L|, |S_R| < \frac{cn}{20} \,.
  \label{eq:end:small-S}
\end{equation}

Finally, we derive more precise bounds on $e(L)$ and $e(R)$, and combine
them with \eqref{eq:end:small-S:eLR+}.  We start with $e(L)$.  By Lemma
\ref{lem:very-high}, every vertex of $L$ can only send at most $\alpha$
edges to $L \setminus S_L$, so the number of edges that are incident to $L
\setminus S_L$ is at most $|L| \alpha$.  All remaining edges in $L$
must have both endpoints in $S_L$, and even if they formed a complete graph
there, their number would be bounded by $\frac{|S_L|^2}{2}$.  Thus
$e(L) < |L| \alpha + \frac{|S_L|^2}{2}$.  Combining this with a similar
estimate for $e(R)$, and with inequality \eqref{eq:end:small-S:eLR+}, we
find that
\begin{align}
  \nonumber
  e(G)
  &<
  (|L| + |R|) \alpha + \frac{|S_L|^2}{2} + \frac{|S_R|^2}{2}
  + \frac{|L| |R|}{2} - |S_L| \left( \frac{cn}{40} \right)
  - |S_R| \left( \frac{cn}{40} \right)
  + \frac{\alpha n}{2} \\
  &=
  \frac{|L| |R|}{2} +
  \frac{3 \alpha n}{2}
  + \frac{|S_L|}{2} \left( |S_L| - \frac{cn}{20} \right)
  + \frac{|S_R|}{2} \left( |S_R| - \frac{cn}{20} \right) \,.
  \label{eq:end:eG+}
\end{align}
Inequality \eqref{eq:end:small-S} shows that the quadratics in $|S_L|$ and
$|S_R|$ are nonpositive.  The maximum possible value of $\frac{|L| |R|}{2}$
is $\frac{n^2}{8}$.  This contradicts our given $e(G) \geq \frac{n^2}{8} +
\frac{3 \alpha n}{2}$, thereby completing the proof.  \end{proof}

\section{Putting everything together}
\label{sec:combine}

Now we finish the proofs by putting the parts together.  Combining the
results of Sections~\ref{sec:great-cut-regularity} and \ref{sec:stability},
we obtain Theorem~\ref{thm:some-reg} which involves an application of
regularity with an absolute constant regularity parameter as input.

\medskip

\begin{proof}[Proof of Theorem \ref{thm:some-reg}] Let $\gamma$ be the
result of feeding $c = \frac{1}{51200}$ into Lemma
\ref{lem:reg:great-cut}, and let $\gamma_0 = 4 \gamma$.  We are given an
$n$-vertex graph with $m \geq \frac{n^2}{8} + \frac{3}{2} \alpha n$
edges, with no $K_4$ and with all independent sets of size at most
$\alpha$, where $\alpha < \gamma_0 n$.  By Lemma \ref{lem:min-degree-1/8},
we may extract a subgraph $G'$ on $n'$ vertices which has at least
$n'\frac{m}{n}\geq n'\left(\frac{n^2}{8}+\frac{3}{2}\alpha n\right)/n \geq \frac{(n')^2}{8} + \frac{3}{2} \alpha n'$ edges, no $K_4$, independence
number at most $\alpha < \gamma n'$, and also minimum degree at least
$\frac{m}{n} > \frac{n'}{8}$.

By Lemma \ref{lem:reg:great-cut}, $G'$ has a cut with at most
$\frac{(n')^2}{51200}$ non-crossing edges.  Finally, the minimum degree
condition of $\frac{n'}{8}$ allows us to apply Lemma
\ref{lem:great-cut=done} with $c = \frac{1}{8}$, as $\frac{1}{51200} =
\frac{(1/8)^2}{800}$ then is the corresponding $c'$.  This completes
the proof.  \end{proof}

\medskip

Next, by combining the results of Sections~\ref{sec:great-cut-no-regularity} and \ref{sec:stability}, we
prove Theorem~\ref{thm:no-reg} without any regularity at all.

\medskip

\begin{proof}[Proof of Theorem \ref{thm:no-reg}]
  By Lemma \ref{lem:min-degree}, we may assume that the minimum degree is
  at least $\frac{n}{4} + (10^{10}-1)\alpha$.  Lemma \ref{lem:indep-cap}
  lets us assume that $\alpha \leq n / (2 \cdot 10^{10})$.  This satisfies
  the conditions of Lemma \ref{lem:great-cut=done} with $c = \frac{1}{4}$,
  so it suffices to show that the max-cut leaves only at most $c' n^2 =
  \frac{(1/4)^2}{800} \cdot n^2 = \frac{n^2}{12800}$ crossing edges.  To
  establish this, we use Lemma \ref{lem:half-sparse=big-cut}, with $c^* =
  \frac{1}{12800}$.  This requires that $\alpha < \frac{c^* n}{50}$, which
  we have, as well as a sparse set $X$ of $\big( \frac{1}{2} - c \big) n$
  vertices, where $c = c^* / 4 = \frac{1}{51200}$.  This is provided by
  Lemma \ref{lem:big-sparse}, which then requires that all degrees are at
  least $\frac{n}{4} + C \alpha$, with $C = \frac{9}{8} \cdot 51200 \cdot
  10^5 + 1 < 6 \cdot 10^9$, as well as requiring that $\alpha \leq 10^{-5}
  \cdot \frac{n}{3 \cdot 51200}$.  As $10^{-5} \cdot \frac{n}{3 \cdot
  51200} \approx 6.5 \cdot 10^{-11}$, our bound from Lemma
  \ref{lem:indep-cap} is indeed sufficient.
\end{proof}

\section{Dependent random choice} \label{sec:DRC}

In this section, we use the good cut discovered by our constant-parameter
regularity approach to find a pair of large disjoint sets of vertices which
has density extremely close to $\frac{1}{2}$.  Then, we introduce our
variant of the dependent random choice technique, and use this to find a large independent set or a $K_4$.

\begin{lemma}
  For any constant $c > 0$, there is a constant $c' > 0$ such that the
  following holds.  Suppose that $\alpha \leq \frac{c^2 n}{1600}$.  Let $G$
  be a graph with at least $\frac{n^2}{8}$ edges, minimum degree at least
  $cn$, no $K_4$, and independence number at most $\alpha$.  Suppose that
  the max-cut of $G$ partitions the vertex set into $L \cup R$ such that
  $e(L) + e(R) \leq c' n^2$.   Then all of the following hold:
  \begin{description}
    \item[(i)] Each of $|L|$ and $|R|$ are between $0.4 n$ and $0.6 n$.
    \item[(ii)] At most $\alpha$ vertices of $L$ have $R$-degree greater
      than $\frac{|R| + \alpha}{2}$.
    \item[(iii)] At most $\alpha$ vertices of $R$ have $L$-degree greater
      than $\frac{|L| + \alpha}{2}$.
    \item[(iv)] Both induced subgraphs $G[L]$ and $G[R]$ have maximum
      degree at most $\big( \frac{120}{c} + 1 \big) \alpha$.
  \end{description}
  \label{lem:1/8:great-cut=small-S}
\end{lemma}

\begin{proof} We may assume $c < 1$, or else there is nothing
to prove.  Let $c' = \frac{c^2}{3200}$.  Now proceed exactly as in the
proof of Lemma \ref{lem:great-cut=done}, and again obtain inequality
\eqref{eq:end:small-S:eLR+}.  Note that along the way, parts (i)--(iii) are
established.  But after reaching \eqref{eq:end:small-S:eLR+}, this time, we
only know $e(G) \geq \frac{n^2}{8}$, so instead of
\eqref{eq:end:small-S:eLR-}, we now have
\begin{equation}
  e(L, R) \geq \frac{n^2}{8} - c' n^2 \,.
  \label{eq:1/8:small-S:eLR-}
\end{equation}
Combining \eqref{eq:end:small-S:eLR+} and \eqref{eq:1/8:small-S:eLR-}, we
obtain the following instead of \eqref{eq:end:small-S:combine}:
\[
    |S_L| \left( \frac{cn}{40} \right) + |S_R| \left( \frac{cn}{40} \right)
  <
  c' n^2 + \frac{\alpha n}{2},
  \]
so
\[
|S_L|
  <
  \frac{40 c' n}{c} + \frac{20 \alpha}{c}
  \leq
  \frac{cn}{40} \,,
  \label{eq:1/8:small-S:combine}
\]
since $c' = \frac{c^2}{3200}$ and $\alpha \leq \frac{c^2 n}{1600}$.  Note
that this is twice as strong as \eqref{eq:end:small-S}.  The same argument
as in the proof of Lemma \ref{lem:great-cut=done} leads again to
\eqref{eq:end:eG+}, which we copy here for the reader's convenience.
\begin{displaymath}
  e(G)
  <
  \frac{|L| |R|}{2} +
  \frac{3 \alpha n}{2}
  + \frac{|S_L|}{2} \left( |S_L| - \frac{cn}{20} \right)
  + \frac{|S_R|}{2} \left( |S_R| - \frac{cn}{20} \right) \,.
\end{displaymath}

This time, we only have $e(G) \geq \frac{n^2}{8}$.  As before, the maximum
possible value of $\frac{|L| |R|}{2}$ is $\frac{n^2}{8}$, so the
nonpositive quadratics in $|S_L|$ and $|S_R|$ are permitted to cost us up
to $\frac{3\alpha n}{2}$ of slack.  However, as we established in
\eqref{eq:1/8:small-S:combine} that $|S_L| < \frac{cn}{40}$, the value of
$\big( \frac{cn}{20} - |S_L| \big)$ is between $\frac{cn}{40}$ and
$\frac{cn}{20}$.  Therefore, we must have
\[
\frac{|S_L|}{2} \left( \frac{cn}{40} \right)
\leq
\frac{|S_L|}{2} \left( \frac{cn}{20} - |S_L| \right)
  <
  \frac{3 \alpha n}{2}
  \]
and hence $|S_L| < \frac{120 \alpha}{c}$.

By Lemma \ref{lem:very-high}, every vertex of $L$ has at most $\alpha$
neighbors in $L \setminus S_L$.  Therefore, every vertex of $L$ has at most
$|S_L| + \alpha < \frac{120 \alpha}{c} + \alpha$ neighbors in $L$.  A
similar argument holds in $R$, establishing part (iv) of this theorem, and
completing the proof.  \end{proof}

\begin{corollary}
  There is an absolute constant $\gamma_0$ such that for every $\gamma <
  \gamma_0$, every $n$-vertex graph with at least $\frac{n^2}{8}$ edges,
  no copy of $K_4$, and independence number at most $\gamma n$, has two
  disjoint subsets of vertices $X$ and $Y$ with $|X| \geq \frac{n}{16}$,
  $|Y| \geq \frac{n}{10}$, and where every vertex of $X$ has $Y$-degree at
  least $\big( \frac{1}{2} - 20000 \gamma \big) |Y|$.
  \label{cor:lots-half-degrees}
\end{corollary}

\begin{proof} Let $\alpha = \gamma n$.  We begin in the same
way as in our proof of Theorem \ref{thm:some-reg} in Section
\ref{sec:combine}, except we use Lemma \ref{lem:1/8:great-cut=small-S}
instead of Lemma \ref{lem:great-cut=done}.  Indeed, let $\gamma_1$ be the
result of feeding $c = \frac{1}{204800}$ into Lemma
\ref{lem:reg:great-cut}, and let $\gamma_0 = \gamma_1 / 4$.  We are given an
$n$-vertex graph with at least $\frac{n^2}{8}$ edges, with no $K_4$ and
with all independent sets of size at most $\alpha$, where $\alpha <
\gamma_0 n$.  By Lemma \ref{lem:min-degree-1/8}, we may extract a subgraph
$G'$ on $n' \geq \frac{n}{4}$ vertices which has at least
$\frac{(n')^2}{8}$ edges, no $K_4$, independence number at most $\gamma n <
\gamma_1 n'$, and also minimum degree at least $\frac{n'}{8}$.

By Lemma \ref{lem:reg:great-cut}, $G'$ has a cut with at most
$\frac{(n')^2}{204800}$ non-crossing edges.  Finally, the minimum degree
condition of $\frac{n'}{8}$ allows us to apply Lemma
\ref{lem:1/8:great-cut=small-S} with $c = \frac{1}{8}$, as
$\frac{1}{204800} = \frac{(1/8)^2}{3200}$ then is the corresponding $c'$.
This gives a bipartition $L \cup R$ of $G'$.  Without loss of generality,
assume that $|L| \geq |R|$, so that $|L| \geq \frac{n'}{2} \geq
\frac{n}{8}$.  Part (i) of that Lemma gives $|R| \geq 0.4 n' \geq
\frac{n}{10}$.  Part (iv) establishes that all degrees in $G[L]$ and
$G[R]$ are at most $\big( \frac{120}{1/8} + 1 \big) \alpha = 961 \alpha$.
Hence
\begin{equation}
  e(L, R)
  \geq
  \frac{(n')^2}{8} - \frac{961 \alpha n'}{2}
  \geq
  \frac{|L| |R|}{2} - \frac{961 \alpha n'}{2}
  \geq
  \frac{|L| |R|}{2} - 961 \alpha |L|
  \,.
  \label{eq:1/8:lots-half:eLR-}
\end{equation}
By part (ii), at most $\alpha$ vertices of $L$ can have $R$-degree greater
than $\frac{|R| + \alpha}{2}$.
Let $Y = R$, and let $X \subset L$ be the vertices that have $R$-degree at
least $\frac{|R|}{2} - 1923.5 \alpha$.  We claim that $|X| \geq
\frac{|L|}{2}$.  Indeed, if this were not the case, then by summing up the
$R$-degrees of the vertices of $L$, we would find
\begin{align*}
  e(L, R)
  &\leq
  \frac{|L|}{2} \left( \frac{|R|}{2} - 1923.5 \alpha \right)
  +
  \left( \frac{|L|}{2} - \alpha \right) \left( \frac{|R|}{2} +
  0.5 \alpha \right)
  +
  \alpha |R| \\
  &=
  \frac{|L| |R|}{2} - 961.5 \alpha |L| + \frac{\alpha |R|}{2} -
  \frac{\alpha^2}{2} \\
  &<
  \frac{|L| |R|}{2} - 961 \alpha |L| \,,
\end{align*}
contradicting \eqref{eq:1/8:lots-half:eLR-}.  Thus $|X| \geq
\frac{|L|}{2} \geq \frac{n'}{4} \geq \frac{n}{16}$.  Finally, observe that
since $\alpha = \gamma n$ and $|R| \geq \frac{n}{10}$ as noted above, we
have
\begin{displaymath}
  1923.5 \alpha
  =
  1923.5 \gamma n
  \leq
  19235 \gamma |R| \,,
\end{displaymath}
and so every vertex of $X$ indeed has $Y$-degree at least $\big(
\frac{1}{2} - 20000 \gamma \big) |Y|$.  \end{proof}

We will present two proofs of Theorem \ref{thm:1/8} that every sufficiently large graph with more
than $\frac{n^2}{8}$ edges contains either a copy of $K_4$, or an
independent set of order $\Omega\big( n \cdot \frac{\log \log n}{\log n}
\big)$. The first proof is shorter. However, the second proof introduces a new twist of the Dependent Random Choice technique, which may find applications elsewhere.

The {\it odd girth} of a graph is the length of the shortest odd cycle
in the graph. Both proofs start by applying  Corollary
\ref{cor:lots-half-degrees} with $\gamma=c\frac{\log \log n}{\log n}$,
where $c>0$ is an absolute constant. In the first proof, we take
$c=10^{-6}$, and show below that the induced subgraph on $X$ has odd girth
at least $1/(40020 \gamma)$. Together with a lemma of Shearer
\cite{Sh}, which implies that every graph on $n$ vertices with odd
girth $2k+3$ has independence number at least $\frac{1}{2} n^{1-1/k}$, we obtain the desired result. Indeed, the odd girth is at least $2k+3$ with $k=10^{-5}\gamma^{-1} = 10\frac{\log n}{\log \log n}$, and hence the independence number is at least $\frac{1}{2}n^{1-1/k}=\frac{1}{2}n(\log n)^{-1/10}$.

To prove a lower bound on the odd girth of the subgraph induced on $X$, consider first a walk $v_1,\ldots,v_t$ in $X$, and let $Y_i$ denote the set of neighbors of $v_i$ in $Y$. Since consecutive vertices on the path are adjacent, Lemma \ref{lem:edge-codegree} implies that $|Y_i \cap Y_{i+1}| \leq \gamma n \leq 10\gamma |Y|$.  Roughly, since the cardinality of each $Y_i$ is almost at least half the order of $Y$, this forces $Y_i,Y_j$ to be nearly complementary if $j-i$ is odd. We next make this claim rigorous. Let $\gamma_k=(10+40020(k-1))\gamma$. We will show by induction on $k$ that
\begin{equation}\label{claimedineq} |Y_{i} \cap Y_{i+2k-1}| \leq \gamma_k |Y|.
\end{equation}
holds for each positive integer $k$.
 The base $k=1$ is satisifed as $|Y_i \cap Y_{i+1}| \leq \gamma n \leq \gamma_1 |Y|$.

The induction hypothesis is $|Y_{i} \cap Y_{i+2k-1}| \leq \gamma_k |Y|$. It follows that $$|(Y_i \cup Y_{i+2k}) \cap Y_{i+2k-1}| \leq |Y_{i} \cap Y_{i+2k-1}|+|Y_{i+2k-1} \cap Y_{i+2k}| \leq \gamma_k |Y|+\gamma n \leq \left(\gamma_k+10\gamma\right)|Y|$$ and
\begin{eqnarray*}
|Y| & \geq & |Y_i \cup Y_{i+2k} \cup Y_{i+2k-1}| \geq |Y_i \cup Y_{i+2k}|+|Y_{i+2k-1}|-|(Y_i \cup Y_{i+2k}) \cap Y_{i+2k-1}| \\ & \geq &  |Y_i \cup Y_{i+2k}|+\left(\frac{1}{2}-20000\gamma\right)|Y|- \left(\gamma_k+10\gamma\right)|Y|.
\end{eqnarray*}
We conclude that $$|Y_i \cup Y_{i+2k}| \leq \left(\frac{1}{2}+20010\gamma+\gamma_k\right)|Y|.$$
Finally, we have \begin{eqnarray*}|Y_i \cap Y_{i+2k+1}| & \leq & |Y_i \setminus Y_{i+2k}|+|Y_{i+2k} \cap Y_{i+2k+1}| =|Y_i \cup Y_{i+2k}| - |Y_{i+2k}|+|Y_{i+2k} \cap Y_{i+2k+1}| \\ & \leq & \left(\frac{1}{2}+20010\gamma+\gamma_k\right)|Y|-\left(\frac{1}{2}-20000\gamma\right)|Y|+\gamma n \leq \left(40020\gamma+\gamma_k\right)|Y| \\ & = & \gamma_{k+1}|Y|.
\end{eqnarray*}
This completes the claimed inequality (\ref{claimedineq}) by induction on $k$. Now suppose the graph has odd girth $2k-1$. So there is a closed walk of that length from a vertex to itself, in which case $Y_{i+2k-1}=Y_i$. Hence, we must have $$\left(\frac{1}{2}-20000\gamma\right)|Y| \leq |Y_i|=|Y_{i+2k-1} \cap Y_i| \leq (10+40020(k-1))\gamma |Y|$$
This implies the odd girth $2k-1$ must satisfy $2k-1 \geq 1/(40020 \gamma)$. This completes the first proof of Theorem \ref{thm:1/8}.

\medskip

We next present the second proof of Theorem \ref{thm:1/8}, starting again with Corollary \ref{cor:lots-half-degrees} . We use a twist of the Dependent Random Choice technique to find either a $K_4$ or a large independent set in $G[X \cup Y]$.  The traditional
technique is to select a random subset $T \subset Y$ by sampling $t$
vertices of $Y$ uniformly at random, with replacement, and then to define
$U \subset X$ as those vertices that are adjacent to every single vertex of
$T$.  Straightforward analysis establishes the following lemma, which was
used in this precise setting by Sudakov \cite{Su-RT} to prove a lower bound
of $n e^{-O(\sqrt{\log n})}$ for this Ramsey-Tur\'an problem.

\begin{lemma}
  [As formulated in \cite{FS-DRC}]
  For every $n$, $d$, $s$, and $k$, every $n$-vertex graph with average
  degree $d$ contains a subset $U$ of at least
  \begin{displaymath}
    \max\left\{
    \frac{d^t}{n^{t-1}} - \binom{n}{s} \left( \frac{k}{n} \right)^t
    :
    t \in \mathbb{Z}^+
    \right\}
  \end{displaymath}
  vertices, such that every subset $S \subset U$ of size $s$ has at least
  $k$ common neighbors.
  \label{lem:FS-DRC}
\end{lemma}

The significance of this lemma is that if it is applied to $G[X \cup Y]$
with $s = 2$ and a suitably chosen $t$, one immediately finds a moderately
sized subset $U$ from which every pair of vertices has many common
neighbors.  Then, either $U$ is an independent set, or it contains an edge
$uv$.  The common neighborhood of $u$ and $v$ is now guaranteed to be
large, and it is either an independent set, or it contains an edge,
creating a $K_4$.  It is worth noting that this approach works even if the
density between $X$ and $Y$ is only bounded away from zero by an
arbitrarily small constant.  However, the lower bound on the independence
number that it gives is only $n e^{-\Theta(\sqrt{\log n})}$.

Yet one might suspect that there is room for improvement, because, for
example, if every vertex of $X$ had $Y$-degree greater than $\big(
\frac{1}{2} + 5 \gamma) |Y| \geq \frac{|Y| + \alpha}{2}$, then it is
already even guaranteed that \emph{every}\/ pair of vertices in $X$ has
common neighborhood larger than $\alpha$, finishing the argument outright.
Our minimum $Y$-degree condition is very close, at $\big( \frac{1}{2} -
20000 \gamma \big) |Y|$.

It turns out that we can indeed capitalize on this, by adjusting
the Dependent Random Choice procedure.  We will still sample $t$ vertices
of $Y$ with replacement, but this time we will place a vertex $u \in X$
into $U$ if and only if at least $\big( \frac{1}{2} + \epsilon \big)t$ of
the sampled vertices are adjacent to $u$.  Relaxing our common adjacency
requirement from $t$ to just over half of $t$ allows us to take many more
vertices into $U$.  In order to analyze this procedure, we will use the
usual Chernoff upper bounds on large Binomial deviations, but we will also
need lower bounds on Binomial tail probabilities.  The second type
guarantees ``dispersion,'' in addition to the usual ``concentration.''

\begin{lemma}
  For any constant $C > 0$, there are $\epsilon_0 > 0$ and $n_0 < \infty$
  such that the following holds for all $\epsilon < \epsilon_0$ and $n >
  n_0$:
  \begin{displaymath}
    \pr{
    \bin{n, \frac{1}{2} - C \epsilon}
    \geq
    \left( \frac{1}{2} + \epsilon \right) n }
    >
    \frac{\epsilon \sqrt{n}}{2} \cdot e^{-n \epsilon^2 (4C^2 + 20C + 16)} \,.
  \end{displaymath}
  \label{lem:dispersion}
\end{lemma}

\begin{proof} Throughout, we will implicitly assume that $n$ is
large and $\epsilon$ is small.  Define
\begin{displaymath}
  p_i
  =
  \binom{n}{i} \left( \frac{1}{2} - C\epsilon \right)^i
  \left( \frac{1}{2} + C\epsilon \right)^{n-i} \,.
\end{displaymath}
Observe that
\begin{displaymath}
  \frac{ \binom{n}{i+1} }{ \binom{n}{i} }
  =
  \frac{ \frac{n!}{(i+1)! (n-i-1)!} }{ \frac{n!}{i! (n-i)!} }
  =
  \frac{n-i}{i+1} \,,
\end{displaymath}
and therefore
\begin{displaymath}
  \frac{p_{i+1}}{p_i}
  =
  \frac{n-i}{i+1} \cdot
  \frac{\frac{1}{2} - C \epsilon}{\frac{1}{2} + C \epsilon} \,.
\end{displaymath}
Since the following inequalities are equivalent:
\begin{align*}
  (n-i) \left( \frac{1}{2} - C \epsilon \right)
  &\leq
  (i+1) \left( \frac{1}{2} + C \epsilon \right) \\
  n \left( \frac{1}{2} - C \epsilon \right)
  - \left( \frac{1}{2} + C \epsilon \right)
  &\leq
  i \left[ \left( \frac{1}{2} + C \epsilon \right)
  + \left( \frac{1}{2} - C \epsilon \right) \right]
  = i \,,
\end{align*}
we know that in particular, $p_i$ is a decreasing sequence for all $i \geq
\frac{n}{2}$.  Hence
\begin{align*}
  \pr{
  \bin{n, \frac{1}{2} - C \epsilon}
  \geq
  \left( \frac{1}{2} + \epsilon \right) n }
  &>
  \epsilon n p_{ (\frac{1}{2} + 2\epsilon)n } \\
  &=
  \epsilon n
  \binom{n}{ (\frac{1}{2} + 2\epsilon)n }
  \left( \frac{1}{2} - C \epsilon \right)^{(\frac{1}{2} + 2\epsilon)n}
  \left( \frac{1}{2} + C \epsilon \right)^{(\frac{1}{2} - 2\epsilon)n} \,.
\end{align*}
By Stirling's formula,
\begin{displaymath}
  \binom{n}{ (\frac{1}{2} + 2\epsilon)n }
  = (1+o(1)) \frac{ \sqrt{2 \pi n} \left( \frac{n}{e} \right)^n }
  { \sqrt{2 \pi (\frac{1}{2} + 2\epsilon)n}
  \left( \frac{(\frac{1}{2} + 2\epsilon)n}{e} \right)^{(\frac{1}{2} +
  2\epsilon)n}
  \sqrt{2 \pi (\frac{1}{2} - 2\epsilon)n}
  \left( \frac{(\frac{1}{2} - 2\epsilon)n}{e} \right)^{(\frac{1}{2} -
  2\epsilon)n}
  } \,,
\end{displaymath}
so our Binomial probability is at least
\begin{displaymath}
  \mathbb{P} >
  \frac{\epsilon \sqrt{n}}{2}
  \cdot
  e^{
  n \left[
  \left( \frac{1}{2} + 2\epsilon \right) \log (1 - 2C\epsilon)
  + \left( \frac{1}{2} - 2\epsilon \right) \log (1 + 2C\epsilon)
  - \left( \frac{1}{2} + 2\epsilon \right) \log (1 + 4\epsilon)
  - \left( \frac{1}{2} - 2\epsilon \right) \log (1 - 4\epsilon)
  \right]
  } \,.
\end{displaymath}

Since $\log (1+x) \leq x$, we have
\begin{displaymath}
  \left( \frac{1}{2} + 2\epsilon \right) \log (1 + 4\epsilon)
  + \left( \frac{1}{2} - 2\epsilon \right) \log (1 - 4\epsilon)
  \leq
  \left( \frac{1}{2} + 2\epsilon \right) (4\epsilon)
  + \left( \frac{1}{2} - 2\epsilon \right) (-4\epsilon)
  \leq
  16 \epsilon^2 \,.
\end{displaymath}
Also, since $\log(1-x) > -2x$ for all sufficiently small positive $x$,
we have
\begin{align*}
  \left( \frac{1}{2} + 2\epsilon \right) \log (1 - 2C\epsilon)
  + \left( \frac{1}{2} - 2\epsilon \right) \log (1 + 2C\epsilon)
  &=
  \frac{1}{2} \log (1 - 4 C^2 \epsilon^2)
  + 2\epsilon \log \frac{1 - 2C\epsilon}{1 + 2C\epsilon} \\
  &>
  -4 C^2 \epsilon^2
  + 2\epsilon \log (1 - 5C\epsilon) \\
  &>
  -4 C^2 \epsilon^2
  - 20 C \epsilon^2 \,.
\end{align*}
This completes the proof.  \end{proof}

\medskip

We are now ready to prove that every sufficiently large graph with more
than $\frac{n^2}{8}$ edges contains either a copy of $K_4$, or an
independent set of size $\Omega\big( n \cdot \frac{\log \log n}{\log n}
\big)$.

\medskip

\begin{proof}[Proof of Theorem \ref{thm:1/8}] Define
\begin{displaymath}
  C = 2000 \,,
  \quad
  K = 4C^2 + 20C + 16 \,,
  \quad
  \gamma = \frac{\log \log n}{200 K \log n} \,,
  \quad
  t = \frac{200 K \log^2 n}{\log \log n} \,,
  \quad
  \text{and}
  \quad
  \epsilon = 10 \gamma \,.
\end{displaymath}
We will show that as long as $n$ is sufficiently large, there must be an
independent set larger than $\gamma n$.  Assume for the sake of
contradiction that there is no $K_4$ and all independent sets have size at
most $\gamma n$.  By Corollary \ref{cor:lots-half-degrees}, there are
disjoint subsets of vertices $A$ and $B$ with $|A| \geq \frac{n}{16}$, $|B|
\geq \frac{n}{10}$, where every vertex of $A$ has $B$-degree at least
$\big( \frac{1}{2} - C \epsilon \big) |B|$.

Select a random multiset $T$ of $t$ vertices of $B$ by independently
sampling $t$ vertices uniformly at random.  Let $U_0 \subset A$ be those
vertices who each have at least $\big( \frac{1}{2} + \epsilon \big) t$
neighbors in $T$.  Note that different vertices of $U_0$ are permitted to
have different neighborhoods in $T$.  Next, for each pair of vertices of
$U_0$ which has at most $\epsilon |B|$ common neighbors in $B$, remove
one of the vertices, and let $U$ be the resulting set.

We claim that $U$ must be an independent set.  Indeed, if not, then there
is an edge in $U$, whose endpoints must have more than $\epsilon |B| \geq
\gamma n$ common neighbors, and hence their common neighborhood contains an
edge, which creates a $K_4$.  It remains to show that $U$ can be large.
Define the random variable $X = |U_0|$, and let $Y$ be the number of pairs
of vertices in $U_0$ that have at most $\epsilon |B|$ common neighbors in
$B$.

We start by estimating $\E{Y}$.  Let $u, v \in A$ be a pair of vertices
whose common neighborhood in $B$ has size at most $\epsilon |B|$.  The only
way in which they could both enter $U_0$ is if both $u$ and $v$ had at
least $\big( \frac{1}{2} + \epsilon \big)t$ elements of $T$ in their
neighborhoods.  In particular, this requires that at least $2 \epsilon t$
elements of $T$ fell in their common neighborhood.  Since elements of $T$
are sampled uniformly from $B$ with replacement, the probability of this is
at most
\begin{displaymath}
  \pr{ \bin{t, \epsilon} \geq 2 \epsilon t }
  <
  e^{-\frac{1}{3} t \epsilon}
  <
  e^{-3 \gamma t} \,.
\end{displaymath}
Here, we used the well-known Chernoff bound (see, e.g., Appendix A of
the book \cite{AlSp} for a reference).  Therefore, by linearity of
expectation,
\begin{displaymath}
  \E{Y}
  <
  n^2 e^{-3 \gamma t}
  =
  \frac{1}{n} \,.
\end{displaymath}

Next, we move to estimate $\E{X}$.  Since every vertex of $A$ has
$B$-degree at least $\big( \frac{1}{2} - C \epsilon \big) |B|$, the
probability that a particular vertex of $A$ is selected for $U_0$ is at
least
\begin{displaymath}
  \pr{ \bin{t, \frac{1}{2} - C \epsilon}
  \geq \left( \frac{1}{2} + \epsilon \right)t }
  >
  \frac{\epsilon \sqrt{t}}{2} \cdot e^{-t \epsilon^2 K}
  =
  \frac{\epsilon \sqrt{t}}{2} \cdot e^{-100 K \gamma^2 t}
  \,,
\end{displaymath}
by Lemma \ref{lem:dispersion}.  As $|A| \geq \frac{n}{16}$, linearity of
expectation gives
\begin{displaymath}
  \E{X}
  >
  \frac{n}{16} \cdot \frac{\epsilon \sqrt{t}}{2} \cdot e^{-100 K \gamma^2 t}
  =
  \frac{n}{16} \cdot 5 \gamma
  \cdot \sqrt{\frac{200 K \log^2 n}{\log \log n}}
  \cdot
  e^{-\frac{1}{2} \log \log n} \,,
\end{displaymath}
which has higher order than $\gamma n$.  Therefore, a final application of
linearity of expectation gives $\E{X - Y} > \gamma n$, and hence there is
an outcome of our random sampling which produces $|U| \geq X - Y > \gamma
n$, so $U$ is too large to be an independent set, a contradiction.
\end{proof}

\section{Quantitative bounds on the Bollob\'as-Erd\H{o}s construction}\label{sectBEgraph}

Recall that Bollob\'as and Erd\H{o}s \cite{BoEr-RT-construct} constructed a
$K_4$-free graph on $n$ vertices with $(1-o(1))\frac{n^2}{8}$ edges with
independence number $o(n)$. The various presentations of the proof of this
result in the literature \cite{BL}, \cite{BL2}, \cite{BoEr-RT-construct},
\cite{EHSSS}, \cite{EHSSS97}, \cite{SS01} do not give quantitative
estimates on the little-$o$ terms. In this section, we present the proof
with quantitative estimates. It shows that the Bollob\'as-Erd\H{o}s graph gives a good lower bound for the Ramsey-Tur\'an numbers in the lower part of the critical window, nearly matching the upper bounds established using dependent random choice. The presentation here closely follows the proof sketched in \cite{SS01}. The next result is the main theorem of this section, which gives the quantitative estimates for the Bollob\'as-Erd\H{o}s construction. Call a graph $G=(V,E)$ on $n$ vertices {\it nice} if it is $K_4$-free and there is a bipartition $V=X \cup Y$ into parts of order $n/2$ such that each part is $K_3$-free.

\begin{theorem} \label{thm:BE}
  There exists some universal constant $C > 0$ such that for every $0 < \e
  < 1$, positive integer $h \geq 16$ and even integer $n \geq
  (C\sqrt{h}/\e)^h$, there exists a nice graph on $n$ vertices, with
  independence number at most $2n e^{-\e \sqrt h / 4}$, and minimum degree
  at least $(1/4 - 2\e) n$.
\end{theorem}

This graph comes from the Bollob\'as-Erd\H{o}s construction, which we
describe now. Let $\mu = \e/\sqrt{h}$. Feige and Schechtman \cite{FS02}
show that, for every even integer $n \geq (C/\mu)^h$, the unit sphere
$\SS^{h-1}$ in $\RR^h$ can be partitioned into  $n/2$ pieces $D_1, \dots,
D_{n/2}$ of equal measure so that each piece has diameter at most
$\mu/4$.\footnote{Lemma 21 in their paper states this for a single value of
$n_0$, but their proof actually shows it for all $n \geq n_0$.}
Choose a vertex $x_i \in D_i$ and an $y_i \in D_i$ for each $i$. Let $X =
\{x_1, \dots, x_{n/2}\}$ and $Y = \{y_1, \dots, y_{n/2}\}$. Construct the
graph $BE(n, h, \e)$ on vertex set $X \cup Y$ as follows:
\begin{enumerate}[(a)]
\item Join $x_i \in X$ to $y_j \in Y$ if $|x_i - y_j| < \sqrt 2 -
  \mu$.
\item Join $x_i \in X$ to $x_j \in X$ if $|x_i - x_j| > 2- \mu$.
  \item Join $y_i \in Y$ to $y_j \in Y$ if $|y_i - y_j| > 2 - \mu$.
\end{enumerate}

Theorem~\ref{thm:BE} then follows from the next four claims.

\begin{claim} \label{cl:BE-K3}
  The subsets $X$ and $Y$ both induce triangle-free subgraphs in
  $BE(n,h,\e)$.
\end{claim}

\begin{claim} \label{cl:BE-K4}
  The graph $BE(n,h,\e)$ is $K_4$-free.
\end{claim}

\begin{claim} \label{cl:BE-ind}
  The independence number of the graph $BE(n,h,\e)$ is at most $2ne^{-\e\sqrt{h}/4}$.
\end{claim}
\begin{claim} \label{cl:BE-deg}
  The minimum degree of the graph $BE(n,h,\e)$ is at least $(1/4 - 2\e)n$.
\end{claim}

\begin{proof}
  [Proof of Claim~\ref{cl:BE-K3}]
  Suppose $x_i, x_j, x_k \in X$ form a triangle. Then
  \[
  0 \leq |x_i + x_j + x_k|^2 = 9 - |x_i - x_j|^2 - |x_i - x_k|^2 -
  |x_j - x_k|^2 < 9 - 3(2-\mu)^2 < 0 \, ,
  \]
 which is a contradiction. So the subgraph induced by $X$ is
 triangle-free, and similarly with $Y$.
\end{proof}

\begin{proof}
  [Proof of Claim~\ref{cl:BE-K4}]
  By Claim~\ref{cl:BE-K3}, any $K_4$ must come from four vertices
  $x,x' \in X$ and $y,y'\in Y$, and if they do form a $K_4$, then
  \begin{align*}
  0 &\leq |x + x' - y - y'|^2
  \\
  & =
  |x - y|^2 + |x - y'|^2 + |x' - y|^2 + |x' - y'|^2 - |x - x'|^2 - |y
  - y'|^2
  \\
  & <
  4(\sqrt 2 - \mu)^2 - 2(2 - \mu)^2
  \\
  &= (2\mu + 8 - 8\sqrt 2) \mu
  \\
  &< 0 \, ,
\end{align*}
which is impossible.
\end{proof}

The following isoperimetric theorem on the sphere shows that, of all
subsets of the sphere of a given diameter, the cap has the largest measure.
It plays a crucial role in the proof, as any independent set which is a
subset of $X$ or $Y$ will have diameter at most $2-\mu$. For a measurable subset $A \subset
  \SS^{h-1}$, let $\lambda(A)$ denote the Lebesgue measure of $A$ normalized so that $\lambda(\SS^{h-1})=1$.

\begin{theorem} [Schmidt \cite{S48}, see also \cite{SS01}]
  Let $\ell \in [0,2]$ and $h$ be a positive integer. If $A \subset
  \SS^{h-1}$ is an arbitrary measurable set with diameter at
  most $\ell$ and $B$ a spherical cap in $\SS^{h-1}$ with diameter
  $\ell$, then $\lambda(A) \leq \lambda(B)$.
\end{theorem}

We have the following corollary of this theorem and a standard concentration of measure inequality for spherical caps (see, e.g., Lemma 2.2 in \cite{B97}).

\begin{corollary}[Corollary 30 in \cite{SS01}] \label{cor:cap-upper}
  Let $\mu \in [0,1)$. If $A \subset \SS^{h-1}$ is any measurable set
  with diameter at most $2 - \mu$, then $\lambda(A) \leq 2e^{-\mu h/2}$.
\end{corollary}

\begin{proof}[Proof of Claim~\ref{cl:BE-ind}]
  We show that the largest independent set contained in each of $X$
  and $Y$ has size at most $ne^{-\e \sqrt{h}/4}$. Let $X_I \subset
  X$ be an independent set. Then the diameter of $X_I$ on $\SS^{h-1}$ is at most
  $2 - \mu$. Let $D_I = \bigcup_{i \in X_I} D_i$. Since the regions $D_i$ all have diameter at most
  $\mu/4$, the diameter of $D_I$ is at most $2 - \mu/2$. By
  Corollary~\ref{cor:cap-upper} we have $\lambda(D_I) \leq 2e^{-\mu h
    /4}$. Since $\lambda(D_I) = |X_I|/(n/2)$, we have $|X_I| \leq n
  e^{-\mu h/4}$.
\end{proof}

The next lemma gives a lower bound on the measure of spherical caps.

\begin{lemma} \label{lem:cap-upper}
  Let $h\geq 5$ be positive integer, and $\e > 0$. Let $B$ be the spherical
  cap in $\SS^{h-1}$ consisting of all points with distance at most $\sqrt
  2 - \frac{\e}{\sqrt h}$ from some fixed point. Then $\lambda(B) \geq
  \frac12 - \sqrt 2 \e$.
\end{lemma}

\begin{proof}
  Let $\delta = \e\sqrt{2/h} - \e^2/(2h)$ so that $(1-\delta)^2 + (1 -
  \delta^2) = (\sqrt 2 - \e/\sqrt h)^2$, and thus $B$ can be taken to be
  $\SS^{h-1} \cap \{x_1 \geq \delta\}$.  Let $A$ be the intersection of the
  ($h$-dimensional) unit ball with the cone determined
  by the origin and the boundary $\SS^{h-1} \cap \{x_1 = \delta\}$ of our
  spherical cap.  Note that $A$ is the convex hull of the origin and the spherical cap $B$. Since we have normalized the total surface area of the
  sphere to be 1, and the cone contains the same fraction of each concentric sphere around the origin, the surface area of this spherical cap $B$ is precisely
  the ratio between the ($h$-dimensional) volumes of $A$ and the entire
  unit ball.  We may lower bound this by replacing $A$ with the simpler
  intersection of the unit ball and the half-space $x_1 \geq \delta$.  It
  therefore suffices to show that the volume of the part of the unit ball
  within the slice $0 \leq x_1 \leq \delta$ is at most $\sqrt{2} \epsilon$
  times the volume of the entire unit ball.  This final ratio is exactly
  \[
  \frac{\displaystyle \int_{0}^{\delta} \left(\sqrt{1-x^2}\right)^{h-1} \ dx}
{\displaystyle \int_{-1}^{1} \left(\sqrt{1-x^2}\right)^{h-1} \ dx} \,,
  \]
  because the $(h-1)$-dimensional intersection between any hyperplane $x_1
  = c$ and the unit ball is always an $(h-1)$-dimensional ball, whose
  measure is an absolute constant multiplied by its radius to the
  $(h-1)$-st power, and the constants cancel between the numerator and
  denominator.

  The numerator is at most $\delta$. Using $1-t \geq e^{-2t}$ for $t
  \in [0,1/2]$, we see that the denominator is at least
  \[
  \int_{-1/2}^{1/2} e^{-(h-1)x^2} \ dx
  = \frac{1}{\sqrt{h-1}} \int_{-\sqrt{h-1}/2}^{\sqrt{h-1}/2} e^{-x^2} \ dx
  \geq \frac{1}{\sqrt{h}} \int_{-1}^{1} e^{-x^2} \ dx
  > \frac{1}{\sqrt{h}}.
  \]
  Thus
  \[
  \lambda(\SS^{h-1} \cap \{0\leq x_1 \leq \delta\})
  \leq \frac{\delta}{1/\sqrt{h}} \leq \sqrt{2} \e. \qedhere
  \]
\end{proof}

\begin{proof}[Proof of Claim~\ref{cl:BE-deg}]
  Take any $x \in X$. We show that $x$ is joined to at least $(1/4 - 2\e)
  n$ vertices in $Y$. The spherical cap containing all points within
  distance at most $\sqrt 2 - 2\mu$ from $x$ has measure at least $1/2 -
  4\e$ by applying Lemma~\ref{lem:cap-upper} with $2\epsilon$. Thus this
  cap must intersect at least $(1/2 - 4\e)n/2$ regions $D_i$, and we have
  $|y_i - x| < \sqrt{2} - \mu$ for each $D_i$ that the cap intersects, so
  that $x y_i$ is an edge of the graph $BE(n,h,\e)$.
\end{proof}

Having completed the proof of Theorem \ref{thm:BE}, we may now easily
obtain Theorem \ref{corbeconstruct}.

\begin{proof}[Proof of Theorem~\ref{corbeconstruct}]
  We apply Theorem~\ref{thm:BE} with $h = \log n / \log \log n$ and $\e$ tending to $0$ sufficiently slowly with $n$, so that $n \geq (C \sqrt{h} / \epsilon)^h$ is
  satisfied for sufficiently large $n$. Theorem \ref{corbeconstruct}
  then follows as an immediate corollary.
\end{proof}

We next formulate another useful corollary of Theorem \ref{thm:BE}. As the
applications in the next section will rely on it, it will be helpful to to
study a variant of the Ramsey-Tur\'an numbers ${\bf RT}(n,K_4,m)$. Recall
that a graph $G=(V,E)$ on $n$ vertices is nice if it is $K_4$-free and
there is a bipartition $V=X \cup Y$ into parts of order $n/2$ such that
each part is $K_3$-free.  Let $S(n,m)$ be the maximum number of edges of a
nice graph on $n$ vertices with independence number less than $m$. Note
that
\begin{displaymath}
  {\bf RT}(n,K_4,m) \geq S(n,m).
\end{displaymath}
This holds because the function $S$ is a more restrictive version of the function ${\bf RT}$. Also note that Theorem \ref{thm:BE} provides a lower bound on $S(n,m)$, as
the Bollob\'as-Erd\H{o}s graph is nice.

For our next corollary, we will pick $h$ to be the largest positive integer
so that $n \geq h^h$, and hence $h > \frac{\log n}{\log \log n}$.  Then,
by picking $\e = 4(\log \log n)h^{-1/2}$, the condition $n \geq
(C\sqrt{h}/\epsilon)^h$ in Theorem \ref{thm:BE} is satisfied.  In this
case, $2ne^{-\epsilon \sqrt{h}/4}=2n/(\log n)$ and the
minimum degree $(1/4-2\epsilon)n$ implies that the graph has at least
$(1/4-2\epsilon)n^2/2=(1/8-\epsilon)n^2$ edges.

\begin{corollary}\label{corbeconstruct2} For $n$ sufficiently large and $\delta = 4(\log \log n)^{3/2}/(\log n)^{1/2}$, we have
${\bf RT}(n,K_4,\delta n) \geq S(n,\delta n) \geq (1/8-\delta)n^2$.
\end{corollary}
Corollary \ref{corbeconstruct2} will be used in combination with the
densifying construction in the next section to prove Theorems
\ref{theorybeyond} and \ref{easythe}, which give lower bounds on
Ramsey-Tur\'an numbers that are significantly larger than the number of
edges in the Bollob\'as-Erd\H{o}s graph.

\section{Above the Bollob\'as-Erd\H{o}s density} \label{sectabove}

If $G$ is a nice graph with edge density less than $1/2$, we will find another nice graph $G'$ on the same vertex set which is a hybrid of $G$ and a complete bipartite graph. The  graph $G'$ is denser than $G$, and its independence number is not much larger than that of $G$. Specifically, with $V_1$, $V_2$ being the triangle-free parts of $G$ of equal size, we will take some $U_1 \subset V_1$ and $U_2 \subset V_2$, and, for $i=1,2$, we add all edges between $U_i$ and $V_{3-i}$ and delete all edges in $V_i$ which contain at least one vertex in $U_i$. Starting with $G$ being a Bollob\'as-Erd\H{o}s graph, we will be able to get a denser nice  graph $G'$ without increasing the independence number too much.

\begin{lemma} \label{hybrid} For positive integers $d$, $m$, $n$ with $n \geq 6$ even and $d \leq n/2$, we have $$S(n,m+d) \geq \left(1-\frac{2d}{n}\right)^2S(n,m)+dn-d^2-n.$$
\end{lemma}
\begin{proof}
Let $G=(V,E)$ be a nice graph on $n$ vertices and $S(n,m)$ edges with independence number less than $m$. So there is a bipartition $V=V_1 \cup V_2$ into parts of order $n/2$ with each part $K_3$-free.

Let $U_i \subset V_i$ for $i=1,2$ be such that $|U_i|=d$ and the induced subgraph of $G$ with vertex set $V \setminus (U_1 \cup U_2)$ has the maximum number of edges. Denote this induced subgraph by $G_0$.
By deleting randomly chosen vertex subsets of $V_1$ and $V_2$ each of order $d$, each edge of $G$ survives in this resulting induced subgraph with probability at least ${n/2-d \choose 2}/{n/2 \choose 2}$. Hence, the number of edges of $G_0$ satisfies \begin{eqnarray*} e(G_0) & \geq & S(n,m){n/2-d \choose 2}/{n/2 \choose 2}=S(n,m)\left(1-\frac{2d}{n}\right)\left(1-\frac{2d}{n-2}\right) \\ & \geq & \left(1-\frac{2d}{n}\right)^2  S(n,m)-n,\end{eqnarray*}
where in the last inequality we used $S(n,m) \leq n^2/3$, which follows from Tur\'an's theorem for $K_4$-free graphs, $d \leq n/2$, and $n \geq 6$.

Modify $G$ to obtain $G'$ by first deleting all edges in $V_i$ which contain at least one vertex in $U_i$, and then adding all edges from $U_i$ to $V_{3-i}$. The number of edges of $G'$ satisfies $$e(G')=e(G_0)+|U_1||V_2|+|U_2||V_1 \setminus U_1|=e(G_0)+dn-d^2 \geq \left(1-\frac{2d}{n}\right)^2  S(n,m)+dn-d^2-n.$$

We next show that $G'$ is nice. Since, for $i=1,2$, the induced subgraph of $G'$ with vertex set $V_i$ is a subgraph of the induced subgraph of $G$ on the same vertex set, then the induced subgraph of $G'$ with vertex set $V_i$ is triangle-free. Assume for the sake of contradiction that $G'$ contained a $K_4$. As $G$ is $K_4$-free, the $K_4$ must contain at least one vertex in $U_1 \cup U_2$. If the $K_4$ contained a vertex $u$ in $U_i$, as all the neighbors of $u$ are in $V_{3-i}$, then the other three vertices in the $K_4$ must be contained in $V_{3-i}$, contradicting that it is triangle-free. Hence $G'$ is $K_4$-free, and hence nice.

As $U_1$ is complete to $U_2$ in $G'$, any independent set in $G'$ cannot contain a vertex in both $U_1$ and $U_2$. As also $|U_1|=|U_2|=d$ and $G$ has independence number less than $m$, the independence number of $G'$ is less than $m+d$.
\end{proof}

We have the following simple corollaries.

\begin{corollary}\label{easyco}
For even $n \geq 6$, if $S(n,m) \geq \left(\frac18-\delta\right)n^2$ with
$n^{-1/2} \leq \delta \leq \frac{1}{4}$, then $S(n,m+2\delta n) \geq
\frac{n^2}{8}$.
\end{corollary}
\begin{proof}
Let $d=2\delta n$. By Lemma \ref{hybrid}, we have
\begin{eqnarray*} S(n,m+d) & \geq & \left(1-\frac{2d}{n}\right)^2S(n,m)+dn-d^2-n \\ & = & \left(1-4\delta\right)^2\left(\frac18 - \delta \right)n^2+2\delta n^2-4\delta^2n^2 -n
\\ & = & \frac{n^2}{8}\left(1+48\delta^2-\frac{8}{n}-128\delta^3\right)
\\ & \geq & \frac{n^2}{8},\end{eqnarray*}
where the last inequality uses $n^{-1/2} \leq \delta \leq \frac{1}{4}$.
\end{proof}

\begin{proof}[Proof of Theorem~\ref{easythe}]
This is an immediate consequence of Corollary \ref{corbeconstruct2}, Corollary \ref{easyco}, and ${\bf RT}(n,K_4,m) \geq S(n,m)$. Indeed, Corollary \ref{corbeconstruct2} states that for $n$ sufficiently large and $\delta = 4(\log \log n)^{3/2}/(\log n)^{1/2}$, we have $S(n,\delta n) \geq (1/8-\delta)n^2$. With this choice of $\delta$ and $m=\delta n$, Corollary \ref{easyco} then implies that, if $n$ is even, we have ${\bf RT}(n,K_4,3\delta n) \geq S(n,3\delta n) \geq \frac{n^2}{8}$. The proof can be easily modified to handle the case $n$ is odd.
\end{proof}

The next corollary allows us to get a lower bound on Ramsey-Tur\'an numbers greater than $n^2/8$.

\begin{corollary}\label{nextabove}
For even $n \geq 6$, if $S(n,m) \geq \left(\frac18-\delta\right)n^2$ and
$\frac{1}{\delta n} \leq a \leq \frac{1}{2}$, then $S(n,m+an) \geq
\frac{n^2}{8}\left(1+4a-4a^2-8\delta\right)$.
\end{corollary}
\begin{proof}
Let $d=a n$. By Lemma \ref{hybrid}, we have
\begin{eqnarray*} S(n,m+d) & \geq & \left(1-\frac{2d}{n}\right)^2 S(n,m)+dn-d^2-n \\ & = & \left(1-2a\right)^2\left(\frac18 - \delta \right)n^2+an^2-a^2n^2 -n
\\ & = & \frac{n^2}{8}\left(1+4a-4a^2-8\delta+32\delta a-32\delta a^2-\frac{8}{n}\right)
\\ & \geq & \frac{n^2}{8}\left(1+4a-4a^2-8\delta\right),\end{eqnarray*}
where the last inequality uses  $\frac{1}{\delta n} \leq a \leq \frac{1}{2}$.
\end{proof}

When $a \gg \delta$, Corollary \ref{nextabove} produces a construction with
substantially more than $\frac{n^2}{8}$ edges.

\begin{proof}[Proof of Theorem~\ref{theorybeyond}]
This is an immediate consequence of Corollary~\ref{corbeconstruct2}, Corollary~\ref{nextabove}, and ${\bf RT}(n,K_4,m) \geq S(n,m)$.  Indeed, Corollary \ref{corbeconstruct2} states that for $n$ sufficiently large and $\delta = 4(\log \log n)^{3/2}/(\log n)^{1/2}$, we have $S(n,\delta n) \geq (1/8-\delta)n^2$. With this choice of $\delta$ and letting $a=\frac{m}{n}-\delta=(1-o(1))\frac{m}{n}$ so that $\delta n +an=m$, Corollary \ref{nextabove} then implies that, if $n$ is even, we have \begin{eqnarray*}{\bf RT}(n,K_4,m) & \geq & S(n,m) \geq \frac{n^2}{8}\left(1+4a-4a^2-8\delta\right) \geq  \frac{n^2}{8}\left(1+\left(\frac{8}{3}-o(1)\right)a\right) \\ & = & \frac{n^2}{8}+\left(\frac{1}{3}-o(1)\right)mn,\end{eqnarray*} where we used $\delta = o(a)$ and
$m \leq \frac{n}{3}$ so that $a \leq \frac{1}{3}$. Note that when $m=o(n)$ we have $a^2=o(a)$, and the bound above improves to ${\bf RT}(n,K_4,m) \geq  \frac{n^2}{8}+\left(\frac{1}{2}-o(1)\right)mn$. The proof can be easily modified to handle the case $n$ is odd.
\end{proof}

\section{Concluding remarks}
\label{sec:conclusion}

In this paper, we solve the Bollob\'as-Erd\H{o}s problem of providing estimates on the independence number of $K_4$-free graphs in the critical window; see Theorem \ref{criticalwindow}. There is still some room to improve the bounds further. For example, in the third part of Theorem \ref{criticalwindow} we showed that for $m$ just $o(n)$, we have ${\bf RT}(n,K_4,m)-n^2/8=\Theta (mn)$, where the implied constants in the lower and upper bound are within a factor $3+o(1)$. It would be interesting to close the gap.

The asymptotic behavior for the Ramsey-Tur\'an numbers for odd cliques were determined by
Erd\H{o}s and S\'os \cite{ErSo} in 1969. They gave a simple proof that if $q$ is odd, then
$${\bf RT}(n,K_q,o(n))=\frac{1}{2}\left(1-\frac{2}{q-1}\right)n^2+o(n^2).$$
Even after the Bollob\'as-Erd\H{o}s-Szemer\'edi result, it still was years before it was generalized by Erd\H{o}s, Hajnal, S\'os, and Szemer\'edi \cite{EHSS83} to all even cliques. They proved, if $q$ is even, then
$${\bf RT}(n,K_q,o(n))=\frac{1}{2}\left(1-\frac{6}{3q-4}\right) n^2+o(n^2).$$
It would be nice to extend the results of this paper concerning the critical window for every even $q$.

Finally, it is quite remarkable that the old construction of Bollob\'as and
Erd\H{o}s can be tweaked to produce lower bounds which nearly reach our new
upper bounds.  Perhaps a further variation using high dimensional geometry
(e.g., changing the underlying space or metric) could further close the
gap.

\vspace{0.1cm}
\noindent {\bf Acknowledgments.} \, We would like to thank David Conlon for helpful comments, and Mathias Schacht for showing us Szemer\'edi's original proof of the Ramsey-Tur\'an result.

\end{document}